\documentclass[12pt]{amsart}
\usepackage{amsmath,amsthm,amssymb,amsfonts,enumerate,color}
\usepackage{amsmath, amsthm, amsfonts, amssymb, color}
\usepackage{mathrsfs}
\usepackage{amsmath, amsthm, amsfonts, amssymb, color}
\usepackage{mathrsfs}
\usepackage{amsfonts, amsmath}
\usepackage{amsmath,amstext,amsthm,amssymb,amsxtra}
\usepackage{txfonts} 
\allowdisplaybreaks
 \usepackage{pgf,tikz}
\usepackage{stmaryrd}
\begin{document}




\newcommand{\norm}[1]{\left\Vert#1\right\Vert}
\newcommand{\abs}[1]{\left\vert#1\right\vert}
\newcommand{\set}[1]{\left\{#1\right\}}
\newcommand{\Real}{\mathbb{R}}
\newcommand{\RR}{\mathbb{R}^n}
\newcommand{\supp}{\operatorname{supp}}
\newcommand{\card}{\operatorname{card}}
\renewcommand{\L}{\mathcal{L}}
\renewcommand{\P}{\mathcal{P}}
\newcommand{\T}{\mathcal{T}}
\newcommand{\A}{\mathbb{A}}
\newcommand{\K}{\mathcal{K}}
\renewcommand{\S}{\mathcal{S}}
\newcommand{\blue}[1]{\textcolor{blue}{#1}}
\newcommand{\red}[1]{\textcolor{red}{#1}}
\newcommand{\Id}{\operatorname{I}}

\newtheorem{thm}{Theorem}[section]
\newtheorem{prop}[thm]{Proposition}
\newtheorem{cor}[thm]{Corollary}
\newtheorem{lem}[thm]{Lemma}
\newtheorem{lemma}[thm]{Lemma}
\newtheorem{exams}[thm]{Examples}
\theoremstyle{definition}
\newtheorem{defn}[thm]{Definition}
\newtheorem{rem}[thm]{Remark}

\numberwithin{equation}{section}

\title[Temperature  of Schr\"odinger operators with initial data in  BMO spaces]
{Characterization of temperatures associated to Schr\"odinger operators with initial data in BMO spaces}

 \author[M.H. Yang and C. Zhang]{Minghua Yang\  and\ Chao Zhang}

 \address{Department of Mathematics\\Jiangxi University of Finance and Economics \\
         Nanchang 330032, PR China}
 \email{ymh20062007@163.com}

 \address{School of Statistics and Mathematics \\
             Zhejiang Gongshang University \\
             Hangzhou 310018, PR China}
 \email{zaoyangzhangchao@163.com}

 \subjclass[2010]{42B35, 42B37, 35J10,  47F05}
\keywords{heat equation,    Schr\"odinger operators,    BMO space, Carleson measure,
 reverse H\"older inequality,     Dirichlet problem.}

\begin{abstract}   Let $\L$ be a Schr\"odinger operator of the form $\L=-\Delta+V$ acting on $L^2(\mathbb R^n)$ where the nonnegative
potential $V$ belongs to  the reverse H\"older class  $B_q$ for some $q\ge n$. Let ${\rm BMO}_{{\mathcal{L}}}(\RR)$ denote the BMO space on $\RR$ associated to the Schr\"odinger operator $\L$.
In this article we will
show that a function  $f\in {\rm BMO}_{{\mathcal{L}}}(\RR)$ is the trace of the solution of
${\mathbb L}u:=u_{t}+\L u=0,  u(x,0)= f(x),$
 where $u$ satisfies  a Carleson-type condition
\begin{eqnarray*}
 \sup_{x_B, r_B}r_B^{-n}\int_0^{r_B^2}\int_{B(x_B, r_B)} \left\{ t|\partial_t   u(x,t)|^2+|\nabla_x   u(x,t)|^2 \right\}{dx dt } \leq C <\infty.
\end{eqnarray*}
Conversely, this Carleson-type  condition characterizes  all the ${\mathbb L}$-carolic functions whose traces belong to
the space ${\rm BMO}_{{\mathcal{L}}}(\RR)$.
This result extends the  analogous characterization found by Fabes and Neri in \cite{FN1}
for   the classical BMO space  of John and Nirenberg.
 \end{abstract}

\maketitle

\section{Introduction and statement of the main result}
Consider the Schr\"odinger  operator
\begin{equation}\label{e1.2}
\L  =-\Delta  +V(x) \ \ \ {\rm on} \ \ L^2(\RR), \ \ \ \  n\geq 3.
\end{equation}
Associated to the nonnegative  potential $V$, we   assume that it is not identically zero and that
 $
V\in B_q$  for  some $q> n/2$,
which by definition means that   $V\in L^{q}_{\rm loc}(\RR), V\geq 0$, and
there exists a  constant $C>0$ such that   the reverse H\"older inequality
\begin{equation}\label{e1.3}
\left(\frac{1}{\abs{B}}\int_BV(y)^q~dy\right)^{1/q}\leq\frac{C}{\abs{B}}\int_BV(y)~dy,
\end{equation}
holds for all   balls $B$ in $\RR.$ The operator   $\L$ is a self-adjoint
operator on $L^2(\RR)$. Hence $\L$ generates the $\L$-heat semigroup $$e^{-t{\L}}f(x)=\int_{\Real^n}\K_t(x,y)f(y)dy,\ f\in L^2(\RR),\ t>0.$$
From the Feynman-Kac formula, it is well-known that  the semigroup kernels ${\mathcal K}_t(x,y)$  of the operators $e^{-t{\L}}$ satisfies
\begin{eqnarray*}
0\leq {\mathcal K}_t(x,y)\leq h_t(x-y)
\end{eqnarray*}
for all $x,y\in\RR$ and $t>0$, where
$$
h_t(x)=(4\pi t)^{-\frac{n}{2}}e^{-\frac{|x|^2}{4t}}
$$
is the kernel of the classical heat semigroup
$\set{{T}_t}_{t>0}=\{e^{t\Delta}\}_{t>0}$ on $\Real^n$. For the classical heat semigroup associated with Laplacian, see \cite{St1970}.

It is well known that the BMO space, i.e. the space
of functions of bounded mean oscillation, is natural substitution to study singular integral at the end-point space $L^{ \infty}(\RR)$.
A celebrated   theorem of  Fefferman and Stein \cite{FS}   states    that
 a BMO function   is the trace of the solution of
 $\partial_{tt}u +\Delta u=0,  u(x,0)= f(x),$
 whenever $u$ satisfies
\begin{eqnarray}\label{ee1.1}
 \sup_{x_B, r_B} r_B^{-n}\int_0^{r_B}\int_{B(x_B, r_B)}|t\nabla  u(x,t)|^2 {dx dt\over t }  \leq C<\infty,
\end{eqnarray}
where { $\Delta=\sum_{i=1}^n\partial_{x_i}^2$ is the Laplace operator and } $\nabla=(\nabla_x, \partial_t)=(\partial_{1},...,\partial_{n}, \partial_t).$  Conversely,
   Fabes, Johnson and Neri \cite{FJN} showed that  condition  above
   characterizes  all the harmonic functions whose traces are in ${\rm BMO}(\RR)$ in 1976.
  The study of this topic has been widely
extended to more general operators such as elliptic operators and Schr\"odinger operators (instead of the Laplacian),  for more general initial data spaces such as Morrey spaces
and for domains other than $\mathbb R^n$ such as Lipschitz domains. For these generalizations,   see \cite{DKP, DYZ, FN1, FN,  HMM, Song}.

In \cite{FN1}, Fabes and Neri further generalized  the above characterization to caloric functions (temperature), that is the authors proved that a BMO function $f$ is the trace of the solution of

\begin{equation*}
\left\{
\begin{aligned}
\partial_{t}u-\Delta u=0&,\ \ \ \  \ \ \ \  x\in \RR, \ t>0, \\
u(x,0)= f(x)&, \ \ \ \ \ \ \ \ x\in \RR,\\
\end{aligned}
\right.
\end{equation*}
 whenever $u$ satisfies
\begin{eqnarray} \label{e1.1}
 \sup_{x_B, r_B} r_B^{-n}\int_0^{r_B^2}\int_{B(x_B, r_B)}  |\nabla_x  u(x,t)|^2 {dx dt }  \leq C<\infty,
\end{eqnarray}
 and, conversely,  the  condition \eqref{e1.1}
   characterizes  all the carolic functions whose traces are in ${\rm BMO}(\RR)$.
The authors in \cite{JX} explored more informations, related to harmonic functions and carolic functions, about this subject.

The main aim of this article is to study a similar characterization to \eqref{e1.1}  for the Schr\"odinger operator with some conditions on its potentials.
In this article, we consider the parabolic Schr\"odinger differential operators
$${\mathbb L}=\partial_{t}+{\L} ,$$
$t>0, x\in\RR$;  see,
for instance, \cite{GJ,TH} and references therein.
For $f\in  L^p(\RR)$, $1\leq p<  \infty,$
it is well known that $u(x,t)=e^{-t{\L}}f(x), t>0, x\in\RR$, is a solution to the heat equation
\begin{eqnarray}\label{el.4}
{\mathbb L}u=\partial_{t}u+{\L} u =0\ \ \ {\rm in }\ {\mathbb R}^{n+1}_+
\end{eqnarray}
with the boundary data $f\in  L^p(\RR)$, $1\leq p<  \infty.$
 The equation  ${\mathbb L}u =0$ is interpreted in the weak sense via a sesquilinear form, that is,
  $u\in {W}^{1, 2}_{{\rm loc}} ( {\mathbb R}^{n+1}_+) $ is a weak solution of ${\mathbb L}u =0$   if it satisfies
$$\int_{{\mathbb R}^{n+1}_+}
{\nabla_x}u(x,t)\cdot {\nabla_x}\psi(x,t)\,dxdt-\int_{{\mathbb R}^{n+1}_+} u(x,t) \partial_t\psi(x,t)dxdt+
 \int_{{\mathbb R}^{n+1}_+} V u\psi \,dxdt=0,\ \ \ \ \forall \psi\in C_0^{1}({\mathbb R}^{n+1}_+).
 $$
 In the sequel,   we call such a function $u$  an ${\mathbb L}$-carolic function associated to the operator ${\mathbb L}$.

As mentioned above,  we  are interested in  deriving the characterization of the solution
 to the equation $ {\mathbb L}u  =0$ in ${\mathbb R}^{n+1}_+$ having boundary values with BMO data.
Following \cite{DGMTZ},   a locally integrable function $f$ belongs to  BMO$_{{\mathcal{L}}}({\mathbb R}^n)$
whenever there is constant $C\geq 0$ so that
\begin{equation} \label{e1.5}
  {1\over |B|}\int_{B}|f(y)- f_{B}|dy\leq C
  \end{equation}
  for every ball $B=B(x, r)$, and
\begin{equation}\label{e1.6}
      {1\over |B|}\int_{B}|f(y)|dy\leq C
\end{equation}
 for every  ball   $B=B(x,r)$ with  $r\geq\rho(x)$.   Here $f_B=|B|^{-1}\int_B f(x) dx$ and
the critical radii above are determined by the function $\rho(x; V)=\rho(x)$ which takes the explicit form
\begin{equation}\label{e1.7}
 \rho(x)=\sup \Big{\{} r>0: \ {1\over r^{n-2}} \int_{B(x, r)} V(y)dy \leq 1 \Big{\}}.
\end{equation}
We define $\|f\|_{{\rm BMO}_{\mathcal{L}}(\RR)}$ to be the smallest
 $C$  in the right hand sides of \eqref{e1.5} and \eqref{e1.6}.
Because of \eqref{e1.6}, this ${\rm BMO}_{{\mathcal{L}}}(\RR)$ space is in fact a proper subspace of  the classical BMO space of John and Nirenberg,
 and it  turns out to be a suitable space in studying the case of the end-point estimates for $p=\infty$ concerning  the boundedness of
 some classical operators associated to $\L$
 such as the Littlewood-Paley square functions, fractional integrals   and Riesz transforms
   (see \cite{BHS2008, BHS2009, DY1, DY2, DGMTZ, HLMMY, MSTZ, MSTZ2}).

Let us introduce a new  function class on the upper half plane $\mathbb{R}_+^{n+1}$.
\begin{defn}[Temperature Mean Oscillation for $\L$]\label{def:1.1}
A $ C^1(\Real_+^{n+1})$-functions $u(x,t)$  belongs to the class  ${\rm TMO_\L}(\Real_+^{n+1})$, if  $u(x,t)$ is
the solution of ${\mathbb L}u=0$  in $\Real_+^{n+1} $ such that
\begin{eqnarray}\label{e1.8}
\|u\|^2_{{\rm TMO_\L}(\Real_+^{n+1})}= \sup_{x_B, r_B}r_B^{-n}\int_0^{r_B^2}\int_{B(x_B, r_B)} \left\{ t|\partial_t   u(x,t)|^2+|\nabla_x   u(x,t)|^2 \right\}{dx dt } <\infty.
\end{eqnarray}
\end{defn}

 The following theorem is the main result of this article.

 \begin{thm}\label{th1.1}
 Suppose $V\in B_q$ for some $q\ge n,$
then we have
 \begin{itemize}
\item[(1)] if $f\in {\rm BMO_\L}(\RR)$, then   the function $u=e^{-t \L}f\in {\rm TMO_\L}(\Real_+^{n+1})$
 with
 $$
 \|u\|_{{\rm TMO_\L}(\Real_+^{n+1})}\le  \|f\|_{{\rm BMO_\L}(\RR)}.$$

\item[(2)]    if $u\in {\rm TMO_\L}(\Real_+^{n+1})$, then    there exists some $f\in {\rm BMO_\L}(\Real^{n})$ such that $u=e^{-t \L}f$,
and
$$
\|f\|_{{\rm BMO_\L}(\RR)}\leq C\|u\|_{{\rm TMO_\L}(\Real_+^{n+1})}
$$
with some constant $C>0$ independent of $u$ and $f$.

 \end{itemize}
\end{thm}
We should mention that for the Schr\"odinger operator $\L$ in \eqref{e1.2},
an important property of the $B_q$ class, proved in \cite[Lemma 3]{Ge}, assures that the condition $V\in B_q$ also implies $V\in B_{q+\epsilon}$
for some $\epsilon>0$
 and that the $B_{q+\epsilon}$ constant of $V$ is controlled in terms of the one of $B_q$ membership. This in particular implies
 $V\in L^q_{\rm loc}(\RR)$ for some $q$ strictly greater than $n/2.$ However,  in general the potential $V$ can be unbounded and does not
 belong to $L^p(\RR)$ for any $1 \le p \le \infty .$  As a model example, we could take $V(x)=|x|^2$.
 Moreover, as
  noted in \cite{Shen}, if $V$ is any nonnegative
 polynomial, then $V$ satisfies the stronger condition
 \begin{eqnarray*}
\max_{x\in B} V(x)\leq\frac{C}{\abs{B}}\int_BV(y)~dy,
\end{eqnarray*}
which implies $V\in B_q$ for every $q\in (1, \infty)$ with a uniform constant.

This article is organized as follows. In Section 2, we recall some preliminary results including
the  kernel estimates of the heat,  
  the $H^1_{\L}(\RR)$ and ${\rm BMO}_{\L}(\RR)$ spaces associated to the Schr\"odinger operators
and certain properties of ${\mathbb L}$-carolic functions.
In  Section 3, we will prove our main result,   Theorem~\ref{th1.1}.
In Section 4,   we will extend the method    for the space ${\rm BMO}_{\L}(\RR)$ in Section 3
 to obtain some generalizations to Lipschitz-type spaces ${\Lambda}_{\L}^{\alpha}(\RR)$
for $\alpha\in (0, 1)$.

Throughout the article, the letters ``$c$ " and ``$C$ " will  denote (possibly different) constants
which are independent of the essential variables.

\vskip 1cm

\section{Basic properties of the heat  semigroups of Schr\"odinger operators}
\setcounter{equation}{0}


In this section, we begin by recalling some basic properties of the critical radii function $\rho(x)$ under the assumption \eqref{e1.3} on $V$
(see Section 2, \cite{DGMTZ}).

\begin{lem}\label{le2.1} Suppose $V\in B_q$ for some $q> n/2.$
There exist $C>0$ and $k_0\geq1$ such that for all $x,y\in\Real^n$
\begin{equation}\label{e2.1}
C^{-1}\rho(x)\left(1+\frac{\abs{x-y}}{\rho(x)}\right)^{-k_0}\leq\rho(y)\leq C\rho(x)
\left(1+\frac{\abs{x-y}}{\rho(x)}\right)^{\frac{k_0}{k_0+1}}.
\end{equation}
In particular, $\rho(x)\sim \rho(y)$ when $y\in B(x, r)$ and $r\leq c\rho(x)$.
\end{lem}
It follows from Lemmas 1.2 and 1.8 in \cite{Shen}  that there is a constant $C_0$ such that
for a nonnegative Schwartz class function $\varphi$ there exists a constant $C$ such that
\begin{equation}\label{e2.2}
\int_{\RR} \varphi_t(x-y)V(y)dy\leq
\left\{
\begin{array}{lll}
Ct^{-1}\left({\sqrt{t}\over \rho(x)}\right)^{\delta}\ \ \ &{\rm for}\ t\leq \rho(x)^2,\\
C\left({\sqrt{t}\over \rho(x)}\right)^{C_0+2-n}\ \ \ &{\rm for}\ t> \rho(x)^2,
\end{array}
\right.
\end{equation}
where $\varphi_t(x)=t^{-n/2}\varphi(x/\sqrt{t}),$   and
$\displaystyle
\delta =2-\frac{n}{q}>0.
$

For the heat kernel ${\mathcal K}_t(x,y)$
of the semigroup
$e^{-t\L}$, we have the following estimates.

\begin{lem}[See {\cite[Proposition 4]{DGMTZ}}] \label{le2.2} Suppose $V\in B_q$ for some $q> n/2.$
For every $N>0$, there exist the constants $C_N$ and $c$ such that for $ x,y\in\Real^n, t >0$, such that
 \begin{itemize}
\item[(i)]
\begin{equation*}
0\leq {\mathcal K}_t(x,y)\leq C_Nt^{-n/2}e^{-\frac{\abs{x-y}^2}{ct}}\left(1+\frac{\sqrt{t}}{\rho(x)}
+\frac{\sqrt{t}}{\rho(y)}\right)^{-N},
\end{equation*}

\item[(ii)] 
\begin{equation*}
 \abs{\partial_t{\mathcal K}_t(x,y)}\leq C_Nt^{-\frac{n+2}{2}}e^{- \frac{\abs{x-y}^2}{ct}}
 \left(1+\frac{\sqrt t}{\rho(x)}+\frac{\sqrt t}{\rho(y)}\right)^{-N}  \ {\rm and}
\end{equation*}
\item[(iii)]
\begin{equation*}
\abs{t\partial_t e^{-t\L}1(x)}\le C_N {\left({\sqrt t/ \rho(x)}\right)^{2-n/q} \over \left(1+{\sqrt t/ \rho(x)}\right)^N }.
\end{equation*}
 \end{itemize}
\end{lem}
In fact, with the same computation as in the proof of \cite[Proposition 4]{DGMTZ}, we have
\begin{equation}\label{eq1}
 \abs{t^m\partial_t^m{\mathcal K}_t(x,y)}\leq C t^{-\frac{n}{2}}e^{- \frac{\abs{x-y}^2}{ct}}
 \left(1+\frac{\sqrt t}{\rho(x)}+\frac{\sqrt t}{\rho(y)}\right)^{-N} .
\end{equation}
Kato-Trotter formula (see for instance   \cite{DZ2002}) asserts that
\begin{align*}
h_t(x-y)- {\mathcal K}_t(x,y)
=  \int_0^t\int_{\Real^n} h_s(x-z)V(z){\mathcal K}_{t-s}(z, y)dzds.
\end{align*}
Then we have the following result.

\begin{lem}[See {\cite[Proposition 4.11]{DZ2002}}]\label{le2.3} Suppose $V\in B_q$ for some $q> n/2.$
There exists a nonnegative Schwartz  function $\varphi$ on
$\Real^n$ such that
\begin{equation*}
\abs{h_t(x-y)-{\mathcal K}_t(x,y)}\leq\left(\frac{\sqrt{t}}{\rho(x)}\right)^{2-n/q}\varphi_t(x-y),\quad x,y\in\Real^n,~t>0,
\end{equation*}
where $\varphi_t(x)=t^{-n/2}\varphi\left(x/\sqrt{t}\right)$.

\end{lem}


Recall that a  Hardy-type space associated to $\L$ was introduced by J. Dziuba\'nski et al.  in \cite{DGMTZ, DZ1999, DZ2002},  defined  by
\begin{equation}\label{e2.7}
 H^1_{\L}(\RR)=\big\{ f\in L^1(\RR): {\mathcal T}^{\ast}f(x)= \sup_{t>0}|e^{-t{\L}}f(x)|\in L^1(\RR) \big\}
\end{equation}
with
$$
\|f\|_{H^1_{\L}(\RR)}=\|{\mathcal T}^{\ast}f\|_{L^1(\RR)}.
$$
 For the above class of potentials,   $H^1_{\L}(\RR)$ admits an atomic characterization, where
cancellation conditions are only required for atoms with small supports.
It can be verified that for   every $m\in{\mathbb N}$, for fixed $t>0$ and  $x\in\RR,$
$\partial_t^m {\mathcal K}_t(x, \cdot)\in H^1_{\L}(\RR)$ with
 \begin{equation}\label{e2.8}
 \|\partial_t^m {\mathcal K}_t(x, \cdot)\|_{H^1_{\L}(\RR)}
 \leq Ct^{-m}.
 \end{equation}
 Indeed, by \eqref{eq1}  we have that for a fixed $y\in \RR,$
 \begin{align*}
&\sup_{s>0} \abs{e^{-s{\L}}\big(t^m\partial_t^m{\mathcal K}_t(\cdot, y)\big)(x)}= \sup_{s>0}\abs{\frac{t^m}{(t+s)^m}\big((t+s)^m\partial_{t+s}^m{\mathcal K}_{t+s}(\cdot, y)\big)(x)}\\
&\leq  C \sup_{s>0} {t^m\over (t+s)^{m+{n\over2}}} e^{- \frac{\abs{x-y}^2}{c(t+s)}}\le C t^m \sup_{s>0} \frac{1}{\abs{x-y}^{n+2m}}e^{- \frac{\abs{x-y}^2}{c(t+s)}} \in L^1(\RR, dx),
 \end{align*}
 which, in combination with the fact that $\partial_t^m {\mathcal K}_t(x, \cdot)=\partial_t^m {\mathcal K}_t(\cdot, x)$,
shows  estimate \eqref{e2.8}.

\begin{lem}\label{le2.5}
Suppose $V\in B_q$ for some $q> n/2.$ Then
 the dual space of $H^1_{\L}(\RR)$ is  ${{\rm BMO}_{\L}(\RR)}$, i.e.,
$$
(H^1_{\L}(\RR))^{\ast}={{\rm BMO}_{\L}(\RR)}.
$$
\end{lem}

 \begin{proof} For the proof, we refer to \cite[Theorem 4]{DGMTZ}. See also \cite{DY2,  HLMMY}.
 \end{proof}

  \smallskip

We now recall a local behavior of solutions to $\partial_{t}u+{\L} u =0$, which was proved in \cite[Lemma 3.3]{WY}, see it also in \cite[Lemma 3.2]{GJ}.
We define parabolic cubes of center $(x,t)$ and radius $r$ by $B_{r}(x,t):=\{(y,s)\in {\RR}\times\Real_+:\abs{y-x}<r,\, t-r^2<s\le t\}=B(x,r)\times(t-r^2,t]$.
And for every $(x,t), (y,s)\in \RR\times (0, \infty)$, we define the parabolic metric: $\abs{(x,t)-(y,s)}=\max\{\abs{x-y}, \abs{s-t}^{1/2}\}$.

\begin{lemma}\label{le2.6} Suppose $0\leq V\in L^q_{\rm loc}(\RR)$ for some $q> n/2.$
  Let $u$ be a weak solution of ${\mathbb L}u=0$
in the parabolic cube $B_{r_0}(x_0,t_0)$.
Then there exists a constant $C=C_n>0$ such that

\begin{eqnarray*}
\sup_{B_{r_0/4}(x_0,t_0)}| u(x,t)| \leq C\Big({1\over r_0^{n+2}}
\int_{B_{r_0/2}(x_0,\, t_0)}| u(x,t)|^2dxdt\Big)^{1/2}.
\end{eqnarray*}
\end{lemma}

 We can prove the following lemma as in \cite[Lemma 2.7]{DYZ},  which is  a similar result, but essentially not the same.

\begin{lem}\label{le2.7}
Suppose $ V\in B_q(\RR)$ for some $q\geq (n+1)/2.$
 Assume that  $u(x,t)\in W_{\rm loc}^{1,2}(\mathbb R^{n+1})$ is a weak solution of  ${\mathbb L}u =0$. Assume that
there is a $d>0$ such that
\begin{eqnarray}\label{e2.9}
\int_{\mathbb R^{n+1}}{|u(x,t)|^2\over 1+|(x,t)|^{n+d}} dxdt\leq C_{d}<\infty.
\end{eqnarray}
Then $u(x,t)=0$ in $\mathbb R^{n+1}$.
\end{lem}

\begin{proof}
Fix an  $R\geq 10$, we let $\varphi\in C_0^{\infty}(B_{3R/4}(0,0))$ such that $0\leq \varphi\leq 1, \varphi=1$ on $B_{5R/8}(0,0)$,
and $|\nabla \varphi|\leq C/R, |\nabla^2 \varphi|\leq C/R^2.$ We have
$$
(\partial_t-\Delta +V)(u\varphi)=u\partial_t\varphi-2\nabla u\cdot \nabla\varphi -u\Delta\varphi \ \ \ {\rm in}\ \mathbb R^{n+1}.
$$
Then we have
$$
(u\varphi)(x,t)=\int_{\mathbb R^{n+1}}{\Gamma_V(x,t;y,s)}\{u\partial_s\varphi-2\nabla u\cdot \nabla\varphi -u\Delta\varphi\} dyds,
$$
where $\Gamma_V(x,t;y,s)$ denotes  the  fundamental solution of $\partial_t-\Delta+V$ in $\mathbb R^{n+1}$.
Hence,  for any $(x,t)\in B_{R/2}(0,0)$,
\begin{eqnarray}\label{e2.10}
|u(x,t)|&\leq& {C\over R}\int_{5R/8\leq |(y,s)|\leq 3R/4} | \Gamma_V(x,t;y,s)| \left(\abs{u(y,s)}+|2\nabla u(y,s)| +{|u(y,s)|\over R}\right) dyds  \nonumber\\
&\leq& {C\over R}\left(1+\frac{2}{R}\right)\left\{ \int_{5R/8\leq |(y,s)|\leq 3R/4}  |\Gamma_V(x,t;y,s) |^2 dyds\right\}^{1/2}
 \left\{ \int_{B_R(0, 0)}  |u(y,s) |^2 dyds\right\}^{1/2},
 \end{eqnarray}
 where we have used the H\"older inequality and Caccioppoli's inequality with parabolic type in \cite[Lemma 3.1]{GJ}.

From the upper bound of $\Gamma_V(x,t;y,s)$ in \cite{Kurata}, (see also in \cite{Shen1999, TH}), we have that for every $k>d/2$ and every  $(x,t)\in B_{R/2}(0, 0)$,
 \begin{eqnarray}\label{e2.11}
 \int_{5R/8\leq |(y,s)|\leq 3R/4}  | \Gamma_V(x,t;y,s)|^2 dyds  &\leq&   C_k\int_{5R/8\leq |y|\leq 3R/4}
 \abs{{1\over \left(1+ {|(x,t)-(y,s)|\over \rho(x)}\right)^k } {e^{-c\frac{\abs{x-y}^2}{t-s}}\over |t-s|^{n/2}} }^2 dyds\nonumber\\
 &\leq& C_k \rho(x)^{2k} R^{2-2k-n}.
 \end{eqnarray}
  Recall that the condition $B_{n/2}$ implies $V\in B_{q_0}$ for some $q_0>n/2$.
  By Lemma~\ref{le2.6}, we have
 $\displaystyle
 |u(y,s)|\leq C\left(\frac{1}{R^{n+2}}\int_{B_R(y,s)}|u(x,t)|^2dxdt\right)^{1/2}.$
 Then we get
 \begin{eqnarray*}
\int_{B_R(0, 0)}  |u(y,s) |^2 dyds   &\leq&  C\int_{B_R(0, 0)}  \left( {1\over R^{n+2}}\int_{B_R(y,s)} |u(x,t)|^2dxdt \right)dyds\\
&\leq&  CR^{n+d-1}\int_{\RR} {|u(x,t)|^2\over 1+|(x,t)|^{n+d}}dxdt
\leq  CC_d { R^{n+d-1}}.
 \end{eqnarray*}
 This, in combination with \eqref{e2.10} and \eqref{e2.11}, yields that for every $(x,t)\in B_{R/2}(0, 0)$,
\begin{eqnarray*}
|u(x,t)|&\leq& C\,  \rho(x)^k\left(1+{2\over R}\right) R^{\frac{d-1}{2}-k}.
 \end{eqnarray*}
Letting $R\to +\infty$, we obtain that $u(x,t)=0$ and therefore, $u=0$  in the whole $\mathbb R^{n+1}.$ The proof is complete.
\end{proof}

\begin{rem} Suppose $ V\in B_q(\RR)$ for some $q>n/2.$ For any $d\geq 0$,  one writes
\begin{eqnarray*}
{ {\mathcal H}_d(\L)}=\Big\{f\in W^{1,2}_{\rm loc}({\mathbb R}^n):
\L f =0 \ {\rm and}\   \ |f(x)|=O(|x|^d) \ \ {\rm as}\ |x|\rightarrow \infty\Big\}
\end{eqnarray*}
and
\begin{eqnarray*}
{ {\mathcal H}_{\L}}=\bigcup_{d:\ 0\leq d<\infty}
{ {\mathcal H}_d(\L)}.
\end{eqnarray*}
By  Lemma~\ref{le2.7}, it follows that
for any $d\geq 0$,
$$
{ {\mathcal H}_{\L}}={ {\mathcal H}_d(\L)}=\big\{0\big\}.
$$
See also   Proposition 6.5 of \cite{DY2}.
\end{rem}

 \medskip

\section{Proof of  the Main Theorem }
\setcounter{equation}{0}

 \bigskip

\subsection{The characterization of ${\rm BMO}_{\L}(\RR)$ in terms of Carleson-type measure}

To prove part (1) of Theorem~\ref{th1.1}, we need  the following lemma.

\begin{lem}\label{le3.8}\cite[Lemma 3.8]{DYZ}
 Suppose $V\in B_q$ for some $q> n.$  Let $\beta=1-{n\over q}$.
  For every $N>0$, there exist    constants $C=C_{N}>0$ and $c>0$ such that
  for all $x,y\in\RR$ and $t>0,$ the semigroup kernels ${\mathcal K}_t(x,y)$,   associated to $e^{-t{\L}}$,
   satisfy the following estimates:
   \begin{itemize}

\item[(i)]
\begin{eqnarray}\label{e3.11}
 | \nabla_x {\mathcal K}_t(x,y)| + | {t} \nabla_x \partial_t{\mathcal K}_t(x,y)|
 \leq C t^{-(n+1)/2}e^{-\frac{\abs{x-y}^2}{ct}}\left(1+\frac{\sqrt{t}}{\rho(x)}
+\frac{\sqrt{t}}{\rho(y)}\right)^{-N},
\end{eqnarray}
\item[(ii)] for $|h|<|x-y|/4,$
\begin{eqnarray*}\label{e3.12}
 | \nabla_x{\mathcal K}_t(x+h,y)- \nabla_x{\mathcal K}_t(x,y)|
 \leq C\left({|h|\over \sqrt{t}}\right)^{\beta}
t^{-(n+1)/2}e^{-\frac{\abs{x-y}^2}{ct}};
\end{eqnarray*}
\item[(iii)]  there is some $\delta>1$ such that
\begin{eqnarray*}\label{e3.13}
 \big|\sqrt t \nabla_x e^{-t{\L}}(1)(x) \big|\le C \min \left\{ \left(\frac{\sqrt t}{\rho(x)}\right)^\delta , \left(\frac{\sqrt t}{\rho(x)}\right)^{-N}\right\}.
\end{eqnarray*}
\end{itemize}
\end{lem}

 \medskip

We  recall that the classical  Carleson
measure is closely related to  the   spaces ${\rm BMO}(\RR)$ and ${\rm BMO}_\L(\RR)$.
But, in this article, we should consider a similar Carleson measure, not the classical one. We say that a
measure $\mu$ defined on ${\mathbb R}^{n+1}_+$ is a $2-$Carleson measure  if there is a positive constant
$c$ such  that for each ball $B$, with radius $r_B$, in ${\mathbb R}^{n}$,
\begin{equation}\label{e3.6}
\mu({\widehat B})\leq c|B|,
\end{equation}
where $\displaystyle {\widehat B}=\{(x,t):x\in B, 0\le t\le r_B^2\}$ is the $2-$tent over $B$.
The smallest bound $c$  in (\ref{e3.6}) is defined to  be the norm of
$\mu$, and is denoted by
$
\interleave\mu\interleave_{2car}$. With a similar argument as in \cite{DGMTZ},  we  know that
 for every $f\in {\rm BMO}_\L(\RR)$,
\begin{equation}\label{e3.18}
 \mu_{\nabla_t, f}(x,t)=t| \partial_te^{-t{\L}}(f)(x)|^2 {dx dt}
\end{equation}
is a 2-Carleson measure with $\interleave \mu_{\nabla_t, f}\interleave_{2car} \leq C\|f\|_{{\rm BMO_\L}(\RR)}$.

Similarly, we know that
$$
 \mu'_{\nabla_t, f}(x,t)=\abs{t^2\frac{de^{-s{\L}}}{ds}\Big|_{s=t^2}(f)(x)}^2 {dx dt\over t}
$$
is a Carleson measure with $f\in \rm BMO_\L(\RR)$. The above measure $\mu'_{\nabla_t, f}(x,t)$ was appeared in \cite[(1.12)]{DGMTZ} and its properties were studied extensively at there.

Let us consider  the square functions
${\mathcal G}f$ and ${\mathcal S}f$ given  by
$$
{\mathcal G}(f)(x)=\Big(\int_0^{\infty}
|\partial_t e^{-t \L}f(x)|^2{dt}\Big)^{1/2}
$$
and
$$
{\mathcal S}(f)(x)=\Big(\int_0^{\infty}
|t\partial_t e^{-t \L}f(x)|^2{dt\over t}\Big)^{1/2}.
$$
By the spectral theory, we have the following identities:
\begin{equation}\label{eqsquare}
\norm{{\mathcal G}(f)}_{L^2(\RR)}=\frac{\sqrt 2}{2}\norm{\L^{1/2}f}_{L^2(\RR)},
\end{equation}
and
\begin{equation}\label{eqsquares}
\norm{{\mathcal S}(f)}_{L^2(\RR)}={1\over 2}\norm{f}_{L^2(\RR)}.
\end{equation}

\begin{proof}[Proof of part (1) of Theorem~\ref{th1.1}]
 Recall that the condition $V\in B_n$   implies $V\in B_{q_0}$ for some $q_0>n$.
From
Lemmas~\ref{le2.2} and  \ref{le3.8}, we see  that $u(x,t)=e^{-t{\L}}f(x)\in C^1({\mathbb R}^{n+1}_+)$.
Let us fix a ball $B=B(x_B, r_B)$. From \eqref{e3.18},
it suffices to  show that there exists a constant $C>0$ such that
\begin{align*}
 \int_0^{r_B^2}\int_B |\nabla_{x}e^{-t { \L}}f(x)|^2{dxdt}
&\le C|B|\norm{f}_{{\rm BMO_\L}(\RR)}^2.
\end{align*}
To do this, we split the function $f$ into local, global, and constant parts as follows
$$f=(f-f_{2B})\chi_{2B}+(f-f_{2B})\chi_{(2B)^c}+f_{2B}=f_1+f_2+f_3,
$$
where $2B=B(x_B, 2r_B)$.

Since the Riesz transform $\nabla \L^{-1/2}$ is bounded on $L^2(\RR)$, by \eqref{eqsquare},   we have
\begin{align*}
 \int_0^{r_B^2}\int_B \abs{\nabla_{x}e^{-t{ \L}}f_1(x) }^2{dxdt}
 &\le \int_0^{r_B^2}\int_{\Real^n} \abs{\nabla_{x}\L^{-{1/2}}\L^{{1/2}}e^{-t {\L}}f_1(x) }^2{dxdt}\\
&\le C \int_0^{\infty}\int_{\RR} \abs{\partial_t e^{-t{ \L}}\left(\L^{-1/2}f_1\right)(x) }^2{dxdt}\\
&\le C\norm{\L^{1/2}\L^{-1/2}f_1}_{L^2(\Real^n)}^2=C\int_{2B}\abs{f(x)-f_{2B}}^2dx\\
&\le C|B|\norm{f}_{{\rm BMO_\L}(\RR)}^2.
\end{align*}

To estimate the global term, we use \eqref{e3.11} in  Lemma~\ref{le3.8}  and then the standard argument as in Theorem 2 of \cite{DGMTZ}   shows that
for $x\in B$ and $t<r^2_B$,
\begin{align*}
&| \nabla_x e^{-t \L}f_2(x) |
 \le C\int_{(2B)^c} \abs{f(y)-f_{2B}}\frac{1}{\abs{x_B-y}^{n+1}}dy\\
&\le C\sum_{k=2}^\infty \frac{1}{(2^kr_B)^{n+1}}\left[\int_{2^kB\backslash 2^{k-1}B}
\abs{f(y)-f_{2^kB}}dy+(2^kr_B)^n\abs{f_{2^kB}-f_{2B}}\right]\\
&\le C\Big({1\over r_B}\Big)\sum_{k=2}^\infty 2^{-k}\left[\norm{f}_{{\rm BMO_\L}(\RR)}
+k\norm{f}_{{\rm BMO_\L}(\RR)}\right]\le C\Big({1\over r_B}\Big)\norm{f}_{{\rm BMO_\L}(\RR)},
\end{align*}
which yields
\begin{align*}
 \int_0^{r_B^2}\int_B \abs{\nabla_{x}e^{-t { \L}}f_2(x) }^2 dxdt\le C|B| r_B^{-2}\int_0^{r_B^2}
1{dt} \norm{f}^2_{{\rm BMO_\L}(\RR)}\leq  C|B|\norm{f}^2_{{\rm BMO_\L}(\RR)}.
\end{align*}
It remains to estimate the constant term $f_3=f_{2B}$, for which we make use of (i) and (iii) of Lemma~\ref{le3.8}.
  Assume first that $r_B\le \rho(x_B)$. By Lemma~\ref{le2.1}, $\rho(x)
\sim \rho(x_B)$ for $x\in B$, we have
 \begin{eqnarray}\label{e3.19}
 \int_0^{r_B^2}\int_B  |\nabla_{x}e^{-t{\L}}f_3(x)|^2{dxdt}
&\leq&
 {\abs{f_{2B}}^2}  \int_0^{r_B^2}\int_B  \big(\sqrt t/\rho(x)\big)^{2\delta} {dxdt\over t}\nonumber\\
&\leq&
 C|B|\abs{f_{2B}}^2   \big(r_B/\rho(x_B)\big)^{2\delta}\nonumber\\
 &\leq&
 C|B|\|f\|_{{\rm BMO_\L}(\RR)}^2\Big(1+\log {\rho(x_B)\over r_B}\Big)^2   \big(r_B/\rho(x_B)\big)^{2\delta}\nonumber\\
  &\leq&
 C|B|\|f\|_{{\rm BMO_\L}(\RR)}^2.
\end{eqnarray}
Suppose finally that $r_B> \rho(x_B)$,  we use an argument as in Theorem 2 of \cite{DGMTZ} to select a finite family
of critical balls $\set{Q_k}$ such that $B\subset \cup Q_k$ and $\sum\abs{Q_k}\le \abs{B}$.
 Then, using the fact that  $\abs{f_{2B}}\le \norm{f}_{\rm BMO_\L(\RR)},$
 we can bound the left hand side of \eqref{e3.19} by
\begin{align*}
&C{\norm{f}_{\rm BMO_\L}^2 }\sum_k\left(\int_0^{\rho(x_k)^2}\int_{Q_k}
\left({\sqrt t\over {\rho(x_k)}}\right)^{2\delta}{dxdt\over t}+\int_{\rho(x_k)^2}^\infty
\int_{Q_k}\left(\frac{\sqrt t}{\rho(x_k)}\right)^{-2N}{dxdt\over t}\right)\\
&\le  C{\norm{f}_{{\rm BMO_\L}(\RR)}^2 }\sum_k \abs{Q_k}\le C|B|\norm{f}_{{\rm BMO_\L}(\RR)}^2,
\end{align*}
which establishes the proof of part (1) of Theorem 1.1.
 \end{proof}

\bigskip

\subsection{Existence  of boundary values of ${\mathbb L}$-carolic functions}
In this section, we will give the proof of part (2) of Theorem \ref{th1.1}.

 First, we need some lemmas for  preparation.
\begin{lem}\label{le3.1} For every  $u\in {\rm TMO_\L}(\Real_+^{n+1})$ and
 for every $k\in{\mathbb N}$, there exists a constant $C_{k,n}>0$ such that
\begin{equation*} \label{dd}
\int_{\RR}{|u(x,{1/k})|^2\over (1+|x|)^{2n}}  dx\leq C_{k,n} <\infty,
\end{equation*}
hence $u(x, 1/k)\in L^2((1+|x|)^{-2n}dx)$. Therefore for  all $k\in{\mathbb N}$, $e^{-t{\L}}(u(\cdot, {1/k}))(x)$ exists
everywhere in ${\mathbb R}^{n+1}_+$.
\end{lem}

\begin{proof}  Since $u\in C^{1}({\mathbb R}^{n+1}_+)$, it  reduces to show that for every $k\in{\mathbb N},$
\begin{eqnarray}\label{e3.1}
\int_{|x|\geq 1} {|u(x,{1/k})- u(x/|x|, 1/k) |^2 \over (1+|x|)^{2n}} dx\leq C_k\|u\|^2_{{\rm TMO_\L}(\Real_+^{n+1})}<\infty.
 \end{eqnarray}
To do this, we write
\begin{align*}
&\hspace{-0.3cm}u(x, 1/k)- u(x/|x|, 1/k)\\&=\big[u(x, 1/k)- u(x, |x|)\big]
 +\big[u(x, |x|)-u(x/|x|, |x|)\big] +\big[u(x/|x|, |x|)-u(x/|x|, 1/k)\big].
 \end{align*}
Let
  \begin{eqnarray*}
 I=  \int_{|x|\geq 1 } {|u(x, 1/k)- u(x, |x|) |^2 \over (1+|x|)^{2n}} dx,
 \end{eqnarray*}
  \begin{eqnarray*}
 II=  \int_{|x|\geq 1 } {|u(x, |x|)-u(x/|x|, |x|) |^2 \over (1+|x|)^{2n}} dx,
 \end{eqnarray*}and
  \begin{eqnarray*}
 III=  \int_{|x|\geq 1 } {|u(x/|x|, |x|)-u(x/|x|, 1/k) |^2 \over (1+|x|)^{2n}} dx.
 \end{eqnarray*}

For $\abs{x}\ge 1$ and $t>0$, let $r^2=t/4$. We use  Lemma~\ref{le2.6} for $\partial_t u$ and Schwarz's inequality to obtain
 \begin{eqnarray}\label{e3.2}
 \big| \partial_t u(x, t)\big| &\leq& C\Big({1\over r^{n+2}} \int_{t-r^2}^t\int_{B(x, r)}
 | \partial_s u(y, s) |^2 {dyds } \Big)^{1/2}\nonumber\\
  &\leq& C\Big({1\over t^{\frac{n}{2}+2}} \int_{t-r^2}^t\int_{B(x, \sqrt t/2)}
 s| \partial_s u(y, s) |^2 {dyds } \Big)^{1/2}\nonumber\\
  &\leq& C t^{-1}\Big({1\over \abs{B(x, \sqrt t)}} \int_{0}^{t}\int_{B(x,\sqrt t)}
  |  s\partial_s u(y, s) |^2 {dyds\over s } \Big)^{1/2}
  \nonumber\\
  &\leq&  C  t^{-1}\|u\|_{{\rm TMO_\L}(\Real_+^{n+1})},
  \end{eqnarray}
which gives
 \begin{eqnarray}\label{e3.3}
 | u(x, 1/k)-u(x, |x|) |  =  \Big| \int_{1/k}^{|x|}   \partial_t u(x, t) dt \ \Big| \leq C{\log(k|x|)} \|u\|_{{\rm TMO_\L}(\Real_+^{n+1})}.
 \end{eqnarray}
It follows that
   \begin{eqnarray*}
 I+III
 &\leq& C\|u\|^2_{{\rm TMO_\L}(\Real_+^{n+1})} \int_{|x|\geq 1 } {1\over (1+|x|)^{2n}} \log^2(k |x|)   dx  \\
 &\leq& C  \|u\|^2_{{\rm TMO_\L}(\Real_+^{n+1})}.
  \end{eqnarray*}

For the term $II,$  we  have that for any $x\in \RR,$
 $$
  u(x, |x|)- u(x/|x|, |x|)  =\int_1^{|x|}  D_r u(r\omega, |x|)  dr, \ \ \ \ x=|x| \omega.
 $$
 Let $B=B(0, 1)$  and  $2^mB=B(0, 2^m)$.
Note that for every $m\in{\mathbb N}$,  we have
  \begin{align*}
   \int_{2^mB\backslash 2^{m-1}B} \left| \int_1^{|x|}\left|    D_r  u(r\omega, |x|) \right|   dr \right|^2 dx
  &=     \int_{2^{m-1}}^{2^{m}} \int_{|\omega|=1} \left| \int_1^{\rho}     D_r u(r\omega, \rho)    dr \right|^2 \rho^{n-1}  d\omega d\rho   \nonumber\\
      &\leq   2^{mn-m} \int_{2^{m-1}}^{2^{m}} \int_{|\omega|=1}\int_1^{ 2^{m}}    |  D_r u(r\omega, \rho)  |^2 dr d\omega d\rho   \nonumber\\
	  &\leq    2^{mn-m} \int_{2^{m-1}}^{2^{m}} \int_{2^mB\backslash B}   |  \nabla_y u(y, t)  |^2 |y|^{1-n} dy dt\nonumber\\
	  &\leq    2^{mn-m} \int_{2^{m-1}}^{2^{m}} \int_{2^mB}   |  \nabla_y u(y, t)  |^2  dy dt,
  \end{align*}
  which gives
  \begin{align*}
 \int_{2^mB\backslash 2^{m-1}B}   | u(x, |x|) - u(x/|x|, |x|) |^2    dx
    &\leq     C2^{2mn-m} \left( {1\over |2^mB|}  \int_{0}^{2^{2m}} \int_{2^mB }    |  \nabla_y u(y, t)  |^2
	{dydt}   \right) \nonumber
	 \\
    &\leq  C  2^{2mn-m}\|u\|^2_{{\rm TMO_\L}(\Real_+^{n+1})}.
  \end{align*}
Therefore,
 \begin{align*}
II
  &\leq C \sum_{m=1}^{\infty}  {1\over 2^{2mn}}
   \int_{2^mB\backslash 2^{m-1}B}    | u(x, |x|) - u(x/|x|, |x|) |^2    dx\leq C\|u\|^2_{{\rm TMO_\L}(\Real_+^{n+1})}.
  \end{align*}
  Combining estimates of $I, II$ and $III$, we have obtained \eqref{e3.1}.

  Note that by Lemma~\ref{le2.2},  if
  $V\in B_q$ for some $q> n/2$, then  the semigroup kernels ${\mathcal K}_t(x,y)$, associated to $e^{-t{\L}}$,
  decay
faster than any power of $1/|x-y|$.
Hence,    for  all $k\in{\mathbb N}$, $e^{-t{\L}}(u(\cdot, {1/k}))(x)$ exists
everywhere in ${\mathbb R}^{n+1}_+$.
This completes the proof.
  \end{proof}

\begin{lem}\label{le3.2} For every  $u\in {\rm TMO_\L}(\Real_+^{n+1})$,
we have that for every $k\in{\mathbb N}$,
\begin{equation*}
u(x, t+{1/k})=e^{-t{\L}}\big(u(\cdot, {1/k})\big)(x), \ \ \ \ x\in\RR, \  t>0.
\end{equation*}
\end{lem}

 \begin{proof}
 Since $u(x, \cdot)$ is continuous on $\Real_+$, we have that $\lim_{t\to 0^+}  u(x,t+{1}/{k})= u(x,{1}/{k}).$
Let us first show that for every $k\in{\mathbb N}$,
 \begin{equation}\label{e3.4}
 \lim_{t\to 0^+}e^{-t \L}(u(\cdot, 1/k))(x)=u(x,1/k),  \ \   x\in \Real^n.
\end{equation}
One writes
\begin{eqnarray*}
e^{-t \L}(u(\cdot, 1/k) )(x)&=&
  e^{-t \L}  (u(\cdot, 1/k)1\!\!1_{|x-\cdot|> 1})(x)\\[1pt]
 & &+
\big(e^{-t \L}- e^{t {\Delta}}\big) (u(\cdot, 1/k)1\!\!1_{|x-\cdot|\leq 1})(x)+
 e^{t {\Delta}}  (u(\cdot, 1/k)1\!\!1_{|x-\cdot|\leq 1})(x) \\[1pt]
&= &
  I(x,t)+II(x,t)+III(x,t).
\end{eqnarray*}
 By Lemma~\ref{le3.1}, we have that $u(x, 1/k)\in L^1((1+|x|)^{-2n}dx)$.
From Lemma~\ref{le2.2}, estimates of  the   semigroup kernels ${\mathcal K}_t(x,y)$, associated to $e^{{-t\L}}$,   show  that
\begin{eqnarray*}
\abs{I(x,t)}&\leq&  C_n\int_{|x-y|> 1}t^{-n/2}e^{-\frac{\abs{x-y}^2}{ct}} \left(1+\frac{\sqrt t}{\rho(x)}+\frac{\sqrt t}{\rho(x)}\right)^{-n}|u(y, 1/k)|dy\\
 &\le& C_n e^{-\frac{1}{ct}}\rho(x)^n\int_{|x-y|> 1} \frac{1}{1+| x-y|^{2n}}  \abs{u(y, 1/k)}  dy\\
 &\le& C_ne^{-\frac{1}{ct}}\rho(x)^n\int_{\RR} \frac{1}{1+| y|^{2n}}  \abs{u(y, 1/k)}  dy,
\end{eqnarray*}
hence $
 \lim_{t\to 0^+}I(x,t)=0.$

For the term $II(x,t)$, it follows from Lemma~\ref{le2.3}  that  there exists a nonnegative Schwartz  function $\varphi$
 such that
 \begin{eqnarray*}
\abs{II(x,t)} &\leq&  C  \int_{|x-y|\leq 1} \abs{{\mathcal K}_t(x,y)-h_t(x-y)}  |u(y, 1/k)| dy \\
&  \leq& C \| u(\cdot, 1/k)  \|_{L^{\infty}( B(x, 1))}\left({\sqrt{t}\over \rho(x)}\right)^{2-n/q} \int_{\RR}
 \varphi_t(x-y)    dy  \\
&\leq&C \left({\sqrt{t}\over \rho(x)}\right)^{2-n/q}\| u(\cdot, 1/k)  \|_{L^{\infty}( B(x, 1))},
\end{eqnarray*}
which implies that
$
 \lim_{t\to 0^+}II(x,t)=0.$
  Finally, we can follow a standard  argument as in the proof of Theorem 1.25, Chapter 1 of \cite{SW} to
 show  that for every  $x\in \Real^n,$
$
\lim_{t\to 0^+}e^{t {\Delta}}  (u(\cdot, 1/k)1\!\!1_{|x-\cdot|\leq 1})(x)= u(x,{1}/{k}),
$
and so \eqref{e3.4} holds.

In order to prove our lemma, we can use Lemma~\ref{le2.7}. Set $w(x,t)=e^{-t \L}(u(\cdot, 1/k))(x)-u(x,t+1/k)$.  The function  $w$ satisfies ${\mathbb L}w =0, w(x,0)=0$.
Define,
 \begin{align*}
{\overline w}(x,t)=\left\{
 \begin{array}{rrl}
  w(x,t),&\ \ \ &t\ge 0,\\[6pt]
  w(x,-t),& \ \ &t<0.
 \end{array}
 \right.
 \end{align*}
Then ${\overline w}$ satisfies
 $$
{\overline {\mathbb L}} {\overline w}(x,t)=0, \ \ \ \ \ \ (x,t)\in  {\mathbb R}^{n+1},
 $$
where  ${\overline {\mathbb L}}$ is an extension operator of ${ {\mathbb L}}$ on $\Real^{n+1}$.
Observe that if $V(x)\in B_q(\Real^{n})$ with $q> n/2,$ then it can be verified
that  $V(x,t)=V(x)\in B_q$ on $\Real^{n+1}$.
Next, let us verify \eqref{e2.9}. One writes
\begin{eqnarray*}
\int_{{\mathbb R}^{n+1}} {|{\overline w}(x,t)|^2\over 1+ |(x,t)|^{6(n+1)}}dxdt
& \leq& C\Big(\int_{{\mathbb R}^{n+1}_+} {|e^{-t \L}(u(\cdot, 1/k))(x)|^2\over 1+ (|x|^2+t)^{3(n+1)}}dxdt \\
&& +\int_{{\mathbb R}^{n+1}_+} {|u(x, 1/k)|^2\over 1+ (|x|^2+t)^{3(n+1)}}dxdt \\
&& +\int_{{\mathbb R}^{n+1}_+} {|u(x, t+1/k)-u(x, 1/k)|^2\over 1+ (|x|^2+t)^{3(n+1)}}dxdt\Big)=C(IV+V+VI). \\
\end{eqnarray*}
Observe that   if $t\geq 1,$ then  by Minkowski inequality and
  Lemma~\ref{le2.2},
\begin{eqnarray*}
 &&\left(\int_{\RR} {|e^{-t \L}(u(\cdot, 1/k))(x)|^2\over (1 + |x|^2)^{\frac {3n+1}{2}}}   dx\right)^{1/2}\\
 & =&  \left(\int_{\RR} {\abs{\int_{\RR}\K_t(x,y)u(y, 1/ k)dy}^2 }   \frac{dx}{ (1 + |x|^2)^{\frac {3n+1}{2}}}\right)^{1/2}\\
 &\le &\int_{\RR} \left({\int_{\RR}\abs{\K_t(x,y) }^2  \frac{dx}{ (1 + |x|^2)^{\frac {3n+1}{2}}}}\right)^{1/2}\abs{u(y, 1/ k)} dy\\
 &\leq& C\int_{\RR} \left(\int_{\RR}{1\over (1+ |x|)^{3n+1}}
 t^{-n}\left(1+\frac{\abs{x-y}}{\sqrt t}\right)^{-4n} dx \right)^{1/2} |u(y,1/k)|dy    \\
 &\leq& C\int_{\RR} \left(\int_{\RR}{1\over (1+ |x|)^{3n+1}}
 {t^{n} \over   (1+|x-y|)^{4n}} dx \right)^{1/2} |u(y,1/k)|dy    \\
&\leq& C t^{n/2} \int_{\RR} {|u(y,1/k)|\over  1+ |y|^{2n} } dy\leq   C_kt^{n/2},
\end{eqnarray*}
which gives
\begin{eqnarray*}
IV \leq  C\int_0^{\infty} {1\over (1+t)^{(3n+1)/2}} \left(  \int_{\RR} {|e^{-t \L}(u(\cdot, 1/k))(x)|^2\over (1 + |x|^2)^{\frac {3n+1}{2}}}   dx\right) dt
\leq C_k\int_0^{\infty} {{t^{n}}\over (1+t)^{(3n+1)/2}} dt \leq   C'_k.
\end{eqnarray*}
If $t\leq  1,$ then $IV \leq C_k$ can be verified easier  by using condition $u\in C^1({\mathbb R}^{n+1}_+)$
and Lemma~\ref{le2.2}.
By Lemma~\ref{le3.1}, we have that $V\leq C_k$.  For term $VI$, we use  \eqref{e3.2} to obtain that $  | \partial_t u(x, t) |
  \leq   C /\sqrt t,
$ and then
 \begin{eqnarray*}
 | u(x, t+1/k)-u(x, 1/k) |  =  \Big| \int_{1/k}^{t+1/k}   \partial_s u(x, s) ds \ \Big| \leq C(\sqrt {t+1/k}-\sqrt {1/k}),
 \end{eqnarray*}
which gives
 \begin{eqnarray*}
VI
& \leq& C \int_{{\mathbb R}^{n+1}_+} { {t+1/k} \over 1+ (|x|^2+t)^{3(n+1)}}dxdt
\leq C_k.
\end{eqnarray*}
  Estimate \eqref{e2.9} then follows readily with $d=5n+6$.

By Lemma \ref{le2.7}, we have that $\overline w\equiv0$, and then $w=0$, that is,
 $$u(x,t+1/k)=e^{-t \L}(u(\cdot, 1/k))(x), \ \ \  x\in\RR, \ t>0.$$
The proof is complete.
\end{proof}

From now on, for any $k\in{\mathbb N}$, we set
$$
u_k(x,t)= u(x, t+{1/k} ).
$$
Following an argument as in \cite[Lemma 1.4]{FJN}, we have

\begin{lem}\label{le3.4}For every  $u\in {\rm TMO_\L}(\Real_+^{n+1})$,
 there exists a constant $C>0$ (depending only on $n$) such that for all  $k\in{\mathbb N},$
\begin{equation}\label{e3.5}
\sup_{x_B, r_B}   r_B^{-n}\int_0^{r_B^2}\int_{B(x_B, r_B)}  | t\partial_t u_k(x,t)|^2 {dx dt\over t}
  \leq C\|u\|^2_{{\rm TMO_\L}(\Real_+^{n+1})}<\infty.
\end{equation}
\end{lem}

 \begin{proof}
Let $B=B(x_B, r_B)$. If $r_B^2\geq 1/k$, then letting $s=t+1/k$, it follows that
 \begin{align*}
 |B|^{-1}\int_0^{r_B^2}\int_B  |  t \partial_t u (x,t+1/k)|^2 {dx dt \over t}
 &\leq  C |B|^{-1} \int_0^{(2r_B)^2}\int_{2B} s| \partial_s u (x,s)|^2 {dx ds }\\
 &\leq C \|u\|^2_{{\rm TMO_\L}(\Real_+^{n+1})}<\infty .
\end{align*}
If $r_B^2< 1/k$, then it follows from Lemma~\ref{le2.6} for   $\partial_{t}u(x, t+{1/k})$ and
 a similar argument as in \eqref{e3.2}   that
 \begin{equation*}
 \big| \partial_{t}u(x, t+{1/k})\big| \leq C \big(t+k^{-1}\big)^{-1/2} \|u\|_{{\rm TMO_\L}(\Real_+^{n+1})}.
  \end{equation*}
 Therefore,
  \begin{align*}
 |B|^{-1}\int_0^{r_B^2}\int_B | t\partial_t u (x,t+1/k)|^2 {dx dt\over t }
 &\leq C  |B|^{-1}\|u\|^2_{{\rm TMO_\L}(\Real_+^{n+1})}   \int_0^{r_B^2}\int_{B}  t\big(t+k^{-1}\big)^{-2} {dx dt}\\
  &\leq C  \|u\|^2_{{\rm TMO_\L}(\Real_+^{n+1})} \,  \Big(k^2 \int_0^{r_B^2}  t { dt}\Big)\\
 &\leq C \|u\|^2_{{\rm TMO_\L}(\Real_+^{n+1})}<\infty
\end{align*}
since
$r_B^2< 1/k.$

By taking the supremum over all balls $B\subset\Real^n,$ we complete the proof of \eqref{e3.5}.
\end{proof}

Letting $f_k(x)=u(x, 1/k), k\in{\mathbb N}$, it follows from Lemma \ref{le3.2}     that
$$
u_k(x,t)=e^{-t{\L}}(f_k)(x), \ \ \ x\in\RR, \ t>0.
$$
And it follows from
Lemma~\ref{le3.4}   that
$$
 \mu_{\nabla_t, f_k}(x,t)=| t\partial_te^{-t{\L}}(f_k)(x)|^2 {dx dt\over t}
$$
is a $2-$Carleson measure with $||| \mu_{\nabla_t, f_k}|||_{2car} \leq C\|u\|^2_{{\rm TMO_\L}(\Real_+^{n+1})}$.

\begin{lem} \label{le3.5} For every  $u\in {\rm TMO_\L}(\Real_+^{n+1})$,
there exists a constant $C>0$ independent of $k$ such that
\begin{equation*}
\|f_k\|_{\rm BMO_\L(\RR)}\leq C<\infty,\ \  \hbox{ for any } k\in \mathbb N.
\end{equation*}
Hence for  all $k\in{\mathbb N}$, $f_k$ is uniformly bounded in  ${\rm BMO_\L(\Real^{n})}$.
\end{lem}

 A similar result of Lemma~\ref{le3.5} was given in \cite[Theorem 2]{DGMTZ}; see also  \cite{DY2,   HLMMY, MSTZ}.
These arguments depend  on three  non-trivial results:    the duality theorem that
${\rm BMO}_{{\mathcal{L}}}(\RR)  = \big( H^1_{\L}(\RR)\big)^{\ast}$,
Carleson inequality on tent spaces (see \cite[Theorem 1]{CMS})
 and
some special properties of the space $H^1_{\L}(\RR)$.
In the sequel, we are going to give a direct proof of Lemma~\ref{le3.5} which is independent of these  results such as the duality of
 $ H^1_{\L}(\RR)$ and ${\rm BMO}_{{\mathcal{L}}}(\RR)$ mentioned above.

To prove  Lemma~\ref{le3.5}, we need to  establish the following Lemmas~\ref{le3.6} and \ref{le3.7}.

Given a function
$f\in L^2((1+|x|)^{-2n}dx)$ and an $L^{2}$ function  $g$  supported on a ball $B=B(x_B, r_B)$, for
 any $ (x,t)\in {\mathbb R}^{n+1}_+$, set
\begin{eqnarray}\label{e3.7}
 F(x,t)=  t\partial_t e^{-t \L}   f(x)\ \ \ {\rm and}\ \ \
G(x,t)= t\partial_t e^{-t \L}  (I-e^{-r^2_B \L} )g(x).
\end{eqnarray}


\begin{lem} \label{le3.6}
   Suppose $f, g, F, G$ are as in (\ref{e3.7}).
 If $f$   satisfies
$$
\interleave \mu_{\nabla_t, f}\interleave^2_{2car}=\sup_{x_B, r_B}
 r_B^{-n}\int_0^{r_B^2}\int_{B(x_B, r_B)} | t \partial_t  e^{-t \L}   f(x)|^2 {dxdt\over t }
 <\infty,
$$
  then there exists a constant $C>0$  such that
\begin{eqnarray}
\int_{{\mathbb R}^{n+1}_+}
|F(x,t) G(x,t)|{dxdt\over t}\leq C |B|^{{1\over 2}}\interleave \mu_{\nabla_t, f}\interleave_{2car}
\|g\|_{{ L}^{2}(B)}.
\label{e3.8}
\end{eqnarray}
\end{lem}

 \begin{proof}
 Given a ball $B=B(x_B, r_B)\subset{\mathbb R}^n$ with radius $r_B$, we  put
$$
T(B)=\{(x,t)\in{\mathbb R}^{n+1}_+: x\in B, \ 0<t<r_B^2\}.
$$
We then write
\begin{align*}
\int_{{\mathbb R}^{n+1}_+}
&\big| F(x,t) G(x,t)\big|{dxdt\over t} \\&=\int_{T(2B)} \big|F(x,t)
G(x,t)\big|{dxdt\over t}+\sum_{k=2}^{\infty}\int_{T(2^{k}B)\backslash T(2^{k-1}B) }
\big|F(x,t) G(x,t)\big|{dxdt\over t}\\
&={\rm A_1} + \sum_{k=2}^{\infty} {\rm A_k}.
\end{align*}
Using the H\"older inequality, the $L^2$-boundedness of square function $\mathcal S$ (see \eqref{eqsquares}) and the $L^2$-boundedness of the operator $e^{-r_B^2\L}$,   we obtain
\begin{align*}
 {\rm A_1}
&\leq \Big\|\left\{\int_0^{(2r_B)^2} |  t\partial_t e^{-t \L}   f(x)  |^2{dt\over t}\right\}^{1/2}\Big\|_{L^2(2B)}
\norm{{\mathcal S} ({
{I}}- e^{-r^2_B \L} )g}_{{ L}^2(\RR)}\\
&\leq C r_B^{{n \over 2}}\interleave \mu_{\nabla_t, f}\interleave_{2car}
\|({
{I}}- e^{-r_B^2 \L} )g\|_{{ L}^2(\RR)}\\
&\le C r_B^{{n \over 2}}\interleave \mu_{\nabla_t, f}\interleave_{2car}\| g\|_{{ L}^{2}(B)}.
\end{align*}
Let us estimate ${\rm A_k}$ for $k=2,3, \cdots.$
Observe that
\begin{align*}
{\rm A}_k
&\leq \Big\|\left\{\int_0^{(2^kr_B)^2}\big|t\partial_t e^{-t \L} f(x)\big|^2{dt\over t}\right\}^{1/2}\Big\|_{ L^2(2^kB)}\\
&\quad \quad \times \Big\|
\left\{\int_0^{(2^kr_B)^2}\big|t\partial_t e^{-t \L}  (I-e^{-r^2_B \L} )g(x)\chi_{T(2^{k}B)\backslash T(2^{k-1}B) }(x,t)\big|^2
{dt\over t}\right\}^{1/2}\Big\|_{{ L}^{2}(2^kB)}\\
&\leq C (2^kr_B)^{n \over2} \interleave \mu_{\nabla_t, f}\interleave_{2car}\times  {\rm B}_k,
\end{align*}
where
\begin{eqnarray*}
{\rm B}_k =\Big{\|}
\Big\{\int_0^{(2^kr_B)^2}\big|t\partial_t e^{-t \L}  (I-e^{-r^2_B \L} )
g(x)\chi_{T(2^{k}B)\backslash T(2^{k-1}B) }(x,t)\big|^2
{dt\over t}\Big\}^{1/2}\Big{\|}_{{ L}^{2}(2^kB)}.
\end{eqnarray*}
To estimate ${\rm B}_k,$ we set
$$
\Psi_{t,s} (\L)h(y)=(t+s)^2\Big({d^2{e^{-r \L}}\over dr^2}
\Big|_{r=t+s}  h\Big)(y).
$$
Note that
$$
 ({ {I}}-e^{-r^2_B \L} )g=\int_0^{r^2_B} \L e^{-s \L}g{{ds}}.
$$ By \eqref{eq1}, we have
\begin{align*}
 {\rm B}_k
& \leq C\Big{\|}\left\{\int_0^{(2^kr_B)^2 }\Big| t\int_0^{r^2_B}
{1\over (t+s)^2} \Psi_{t,s} (\L)g(x)\chi_{T(2^{k}B)\backslash T(2^{k-1}B) }(x,t){ds} \Big|^2  {dt \over t} \right\}^{1/2}\Big{\|}_{{ L}^{2}(2^kB)}\\
&\leq C\Big{\|}\Big\{\int_0^{(2^kr_B)^2} \Big| t\int_0^{r^2_B} \int_{B(x_B,r_B)}
{1\over (t+s)^{{n\over 2}+2}}e^{-c\frac{\abs{x-y}^2} {t+s}} \\&\quad \quad \quad \quad \quad \quad \quad \quad
\quad \quad \quad \quad \quad \quad \quad  \times |g(y)|\chi_{T(2^{k}B)\backslash T(2^{k-1}B) }(x,t)
{dyds} \Big|^2  {dt\over t} \Big\}^{1/2}\Big{\|}_{{ L}^{2}(2^kB)}.
\end{align*}
Note that for   $(x,t)\in T(2^{k}B)\backslash T(2^{k-1}B)$
and $y\in B$, we have that $|x-y|\geq 2^kr_B$.   So
\begin{align*}
{\rm B}_k &\le C\Big{\|}\Big\{\int_0^{(2^kr_B)^2} \Big| t\int_0^{r^2_B} \int_{B(x_B,r_B)}
{1\over \abs{x-y}^{n+4}}
 |g(y)|\chi_{T(2^{k}B)\backslash T(2^{k-1}B) }(x,t)
{dyds} \Big|^2  {dt\over t } \Big\}^{1/2}\Big{\|}_{{ L}^{2}(2^kB)}\\
&\le C\norm{g}_{L^1(B)}\Big{\|}\Big\{\int_0^{(2^kr_B)^2} \Big| t\int_0^{r^2_B}
{1\over (2^k r_B)^{n+4}}
 \chi_{T(2^{k}B)\backslash T(2^{k-1}B) }(x,t)
{ds} \Big|^2  {dt\over t } \Big\}^{1/2}\Big{\|}_{{ L}^{2}(2^kB)}\\
&\le C\norm{g}_{L^1(B)}\frac{r^2_B}{(2^k r_B)^{n+4}}\Big{\|}\int_0^{(2^kr_B)^2}
 \chi_{T(2^{k}B)\backslash T(2^{k-1}B) }(x,t)
\ \  tdt \Big{\|}^{1/2}_{{ L}^{1}(2^kB)}\\
&\leq  C { 2^{(-2-{n\over 2})k}r_B^{-\frac{n}{2}}\norm{g}_{L^1(B)}\le C2^{(-\frac{n}{2}-2)k}\norm{g}_{L^{2}(B)}.}
\end{align*}
Consequently,
\begin{equation*}
{\rm A}_k
\leq C 2^{-2k}  r_B^{{n\over 2}}   \interleave \mu_{\nabla_t, f}\interleave_{2car}\norm{g}_{L^{2}(B)},
\end{equation*}
which implies
\begin{align*}
\int_{{\mathbb R}^{n+1}_+}
|F(x,t) G(x,t)|{dxdt\over t}
&\leq Cr_B^{{n  \over 2}}\interleave \mu_{\nabla_t, f}\interleave_{2car}\|g\|_{{ L}^{2}(B)}
 +
C\sum_{k=2}^{\infty} 2^{-2k}  r_B^{{n  \over 2}}
\interleave\mu_{\nabla_t, f}\interleave_{2car}\|g\|_{{L}^{2}(B)} \\
&\leq  Cr_B^{{n  \over 2}}
\interleave \mu_{\nabla_t, f}\interleave_{2car}\|g\|_{{L}^{2}(B)}
\end{align*}
as desired.
\end{proof}

\begin{lem} \label{le3.7}   Suppose $B, f, g, F, G $ are
defined as in Lemma~\ref{le3.6}. If  $\interleave \mu_{\nabla_t, f}\interleave_{2car}<\infty$, then
we have the equality:
\begin{eqnarray*}
\int_{{\mathbb R}^n} f(x) ({\mathcal {I}}-e^{-r^2_B{\L}})g(x)dx={1\over 4}\int_{{\mathbb R}^{n+1}_+}
F(x,t) G(x,t){dxdt\over t}.
\end{eqnarray*}
\end{lem}

\begin{proof}
The technique of this lemma's proof has been used in lots of papers, for example \cite{DXY, DY2, DGMTZ, MSTZ}, but it is notable to state it at here for completeness.

By Lemma \ref{le3.6}, we know that $\displaystyle\int_{\Real_+^{n+1}}\abs{F(x,t)G(x,t)}{dxdt\over t}<\infty$. By dominated convergence theorem, the following integral converges absolutely and satisfies
\begin{equation*}
I=\int_{{\mathbb R}^{n+1}_+}F(x,t) G(x,t){dxdt\over t}=\lim_{\epsilon\rightarrow 0^+}\int_\epsilon^{1\over \epsilon}\int_{\RR}F(x,t)G(x,t){dxdt\over t}.
\end{equation*}
By Fubini's theorem, together with the commutative property of the semigroup $\{e^{-t\L}\}_{t>0}$, we have
\begin{eqnarray*}
\int_{\RR}F(x,t)G(x,t)dx&=&\int_{\RR}\int_{\RR}t\partial_t\K_t(x,y)f(y)t\partial_t\K_t(x,y)(\mathcal I-e^{-r^2_B\L})g(y)dydx\\
&=&\int_{\RR} f(y)\Big(t\partial_te^{-t\L}\Big)^2(\mathcal I-e^{-r^2_B\L})g(y)dy.
\end{eqnarray*}
 Hence,
 \begin{eqnarray*}
I&=&\lim_{\epsilon\rightarrow 0^+}\int_\epsilon^{1\over \epsilon}\int_{\RR} f(x)(t\partial_te^{-t\L})^2(\mathcal I-e^{-r^2_B\L})g(x){dx dt\over t}\\
&=&\lim_{\epsilon\rightarrow 0^+} \int_{\RR}f(x)\int_\epsilon^{1\over \epsilon}(t\partial_te^{-t\L})^2(\mathcal I-e^{-r^2_B\L})g(x){dtdx\over t}.
\end{eqnarray*}
By \cite[Lemma 7]{DGMTZ}, we can pass the limit inside the integral above. And, by a similar computation of \cite[Lemma 3.7]{MSTZ} with $\beta=1$ , we have
 \begin{eqnarray*}
I&=& \int_{\RR}f(x)\int_0^{\infty} (t\partial_te^{-t\L})^2(\mathcal I-e^{-r^2_B\L})g(x){dtdx\over t}=4\int_{\RR}f(x)(\mathcal I-e^{-r^2_B\L})g(x)dx.
\end{eqnarray*}
This completes the proof.
\end{proof}

\begin{proof}[Proof of Lemma~\ref{le3.5}]
First, we note an equivalent  characterization of $ {\rm BMO}_{\L}(\RR)$ that
   $f\in {\rm BMO}_{\L}(\RR) $ if and only if  $f\in L^2((1+|x|)^{-(n+\epsilon)}dx)$ and
  \begin{equation} \label{e3.9}
  \sup_B \Big( |B|^{-1}\int_{B}|f(x)- e^{-r^2_B{\L}}f(x) |^2dx\Big)^{1/2}\leq C<\infty.
  \end{equation}
This has been proved in \cite[Proposition 6.11]{DY2}
  (see  also \cite{DY1, HLMMY}).

 Now if $\|u\|_{{\rm TMO_\L}(\Real_+^{n+1})}<\infty$, then it follows from Lemma \ref{le3.1} that
 $$
 \int_{\Real^{n}} {|f_k(x)|^2\over 1+|x|^{2n}} dx\leq C_k <\infty.
 $$
 Given an $L^{2}$ function $g$ supported on a ball $B=B(x_B, r_B)$, it follows by  Lemma~\ref{le3.7}  that
 we have
\begin{align*}
  \int_{{\mathbb R}^n}
f_k(x) (I-e^{-r^2_B \L} )g(x)dx
 = {1\over 4}\int_{{\mathbb R}^{n+1}_+}  t \partial_t e^{-t \L}   f_k(x)  \
 t\partial_t e^{-t \L}  (\mathcal I-e^{-r^2_B \L} )g(x) {dxdt\over t}.
\end{align*}
By Lemmas~\ref{le3.6} and ~\ref{le3.4},
\begin{align*}
  \abs{\int_{{\mathbb R}^n}
f_k(x) (\mathcal I-e^{-r^2_B\L} )g(x)dx}  &\leq C|B|^{1/2}
 \interleave \mu_{\nabla_t,  f_k}\interleave_{2car} \|g\|_{{L}^{2}(B)}\\
 &\leq C|B|^{1/2}\|u\|_{{\rm TMO_\L}(\Real_+^{n+1})} \|g\|_{{L}^{2}(B)}.
\end{align*}
  Then the duality argument for $L^2$ shows that
\begin{align*}
\Big(|B|^{-1}\int_B |f_k(x)-e^{-r^2_B \L}f_k(x)|^2dx\Big)^{1/2}
&=|B|^{-1/2}\sup\limits_{\|g\|_{{ L}^{2}(B)}\leq 1}\Big|\int_{{\mathbb R}^n}
(\mathcal I-e^{-r^2_B \L})f_k(x)g(x)dx\Big|\nonumber\\
&=|B|^{-1/2}\sup\limits_{\|g\|_{{ L}^{2}(B)}\leq 1}\Big|\int_{{\mathbb R}^n}
f_k(x)\big(\mathcal I-e^{-r^2_B \L}\big)g(x)dx\Big|\nonumber\\
&\leq C
\|u\|_{{\rm TMO_\L}(\Real_+^{n+1})}
\end{align*}
for some $C>0$   independent of $k.$

 It then follows
 that for all $k\in{\mathbb N}$, $f_k$ is uniformly bounded in
${\rm BMO}_{\rm \L}({\mathbb R}^n)$.
\end{proof}


 \begin{proof}[Proof of part (2) of Theorem~\ref{th1.1}]  Letting $f_k(x)=u(x, 1/k)$,
 it follows by Lemma~\ref{le3.2} that $u(x, t+{1/k} )=    e^{-t{\L}}(f_k)(x)$ and so
 \begin{align}\label{e3.10}
    \L u(x, t+{1/k})=  \L e^{-t{\L}}(f_k)(x).
\end{align}
Then we have the following facts:

 \begin{itemize}

\item[(i)]    $H^1_\L(\Real^{n})$ is a Banach space;

\item[(ii)]
 For each $t>0$ and $x\in\RR$, $\partial^2_t {\mathcal K}_t(x, \cdot)\in H^1_{\L}(\RR)$ with $\|\partial^2_t {\mathcal K}_t(x, \cdot)\|_{H^1_{\L}(\RR)}
 \leq C/t^2$ (see \eqref{e2.8});

  \smallskip

\item[(iii)]
 The duality theorem that
 $ (H^1_\L(\Real^{n}) )^{\ast}= {\rm BMO_\L}(\Real^{n})$(see Lemma~\ref{le2.5}).

 \end{itemize}

 \noindent
  From (i), we use Lemma~\ref{le3.5} and pass  to a subsequence, we have that $f_k\rightarrow f$
  (in weak-${\ast}$ convergence) for some $f\in {\rm BMO}_\L(\Real^{n})$
 such that $\|f\|_{{\rm BMO}_\L(\Real^{n})}\leq C\|u\|_{{\rm TMO_\L}(\Real_+^{n+1})}$. Then by (ii) and (iii),
 we conclude that, for each $(x,t)\in {\mathbb R}_+^{n+1}$, the right-hand side
 of \eqref{e3.10} converges to $\L e^{-t{\L}}(f)(x)$ when $k\to \infty$. On the other hand,
  as $k\rightarrow \infty$, the left-hand side of \eqref{e3.10} converges pointwisely
 to $ \L u(x, t)$.   Hence, for a fixed $t>0,$
 $$\L w(x, t) =0\ \ {\rm in}\ \
 {\mathbb R}^{n},
 $$
 where $w(x,t)=u(x, t)- e^{-t{\L}}(f)(x)$.
  The function  $w$ satisfies ${\mathbb L}w =0, w(x,0)=0$.
Define,
 \begin{align*}
{\overline w}(x,t)=\left\{
 \begin{array}{rrl}
  w(x,t),&\ \ \ &t\ge 0,\\[6pt]
  w(x,-t),& \ \ &t<0.
 \end{array}
 \right.
 \end{align*}
Then ${\overline w}$ satisfies
 $$
{\overline {\mathbb L}} {\overline w}(x,t)=0, \ \ \ \ \ \ (x,t)\in  {\mathbb R}^{n+1},
 $$
where  ${\overline {\mathbb L}}$ is an extension operator of ${ {\mathbb L}}$ on $\Real^{n+1}$.
Observe that if $V(x)\in B_q(\Real^{n})$ with $q\ge n,$ then it can be verified
that  $V(x,t)=V(x)\in B_q$ on $\Real^{n+1}$.
Next, let us verify \eqref{e2.9}.
 One writes
\begin{eqnarray*}
\int_{{\mathbb R}^{n+1}} {|{\overline w}(x,t)|^2\over 1+ |(x,t)|^{2(n+4)}}dxdt
& \leq& C\Big(\int_{{\mathbb R}^{n+1}_+} {|u(x, 1/k)|^2\over 1+ (|x|^2+t)^{(n+4)}}dxdt \\
&& +\int_{{\mathbb R}^{n+1}_+} {|u(x, t)-u(x, 1/k)|^2\over 1+ (|x|^2+t)^{(n+4)}}dxdt+ \int_{{\mathbb R}^{n+1}_+} {|e^{-t{\L}}(f)(x)|^2\over 1+ (|x|^2+t)^{(n+4)}}dxdt\Big)\\
& \leq & 2( I+II+III).
\end{eqnarray*}
By Lemma~\ref{le3.1}, we have that $\displaystyle I\leq C_k\int_{\mathbb R^+}\frac{1}{1+t^4}dt<\infty$.   For term $II$, we use  \eqref{e3.2} to obtain that $  | \partial_t u(x, t) |
  \leq   C /\sqrt t,
$ and then
 \begin{eqnarray*}
 | u(x, t)-u(x, 1/k) |  =  \Big| \int_{1/k}^{t}   \partial_s u(x, s) ds \ \Big| \leq C t^{1/2},
 \end{eqnarray*}
which gives
 \begin{eqnarray*}
II
& \leq& C \int_{\mathbb R^+}\int_{\RR}\frac{t}{1+t^4} {1\over 1+ |x|^{2n} }dxdt
\leq C.
\end{eqnarray*}

Consider the term $III$. For a fixed $t>0$ and $x\in R^n$, we set $f_B=t^{-n}\int_{B(x,t)} f(y)dy.$
It can be verified  by  a standard argument (see \cite[Theorem 5]{DGMTZ}) that
$|e^{-t{\L}}f(x)|
 \leq C\|f\|_{{\rm BMO}_\L(\Real^{n})} +|f_B|.$    By Lemma 2 of  \cite{DGMTZ}, we have that
 $|f_B|\leq C(1+{\rm log}[{\rho(x)/t}])\|f\|_{{\rm BMO}_\L(\Real^{n})}$ if $t<\rho(x)$;
  $|f_B|\leq C \|f\|_{{\rm BMO}_\L(\Real^{n})}$ if $t\geq \rho(x)$. It then follows by Lemma~\ref{le2.1}  that
  there is a constant $k_0\geq 1$ such that
 \begin{eqnarray*}
|e^{-t{\L}}f(x)| \leq  C\|f\|_{{\rm BMO}_\L(\Real^{n})}
\left(1+  {\rm log}\left[1 + {C\rho(0)\over t}\left(1+{|x|\over \rho(0)}\right)^{k_0\over k_0+1}\right] \right)
\end{eqnarray*}
 for every $t>0$ and   $x\in\RR,$
which yields
 $$
 III\leq  C\|f\|^2_{{\rm BMO}_\L(\Real^{n})} \int_{\mathbb R^+}\frac{1}{1+t^4}\int_{\RR} {1\over 1+ |x|^{2 n}}
 \left(1+  {\rm log}\left[ 1 + {C\rho(0)\over t}\left(1+{|x|\over \rho(0)}\right)^{k_0\over k_0+1}\right] \right)^2dxdt  \le C<\infty.
 $$
  Estimate \eqref{e2.9} then follows readily with $d=n+8$.

By Lemma~\ref{le2.7}, we have that
 $u(x,t)=e^{-t{\L}}(f)(x)$ with   $f\in {\rm BMO}_\L(\Real^{n}).$  The proof of part (2) of Theorem~\ref{th1.1}
 is complete.
\end{proof}

\medskip

\section{The spaces ${\rm TMO}^\alpha_{\L}({\mathbb R}^{n+1}_+)$ and their characterizations}
\setcounter{equation}{0}

In this section we will extend the method   for the space ${\rm BMO}_{\L}(\RR)$ in Section 3
 to obtain some generalizations to  Lipschitz-type spaces $\Lambda^{\alpha}_{\L}(\RR)$ with
  $0<\alpha<1$ (see \cite{BHS2008}).

 Let us recall that a locally integrable function $f$ in
 $\Lambda^{\alpha}_{\L}(\RR), 0<\alpha<1,$ if  there exists
 a constant $C>0$ such that
 \begin{eqnarray}\label{e4.1}
 |f(x)-f(y)|\leq C |x-y|^{\alpha}
\end{eqnarray}
 and
 \begin{eqnarray}\label{e4.2}
 |f(x)|\leq C \rho(x)^{\alpha}
\end{eqnarray}
 for all $x,y\in\RR.$  The norm of $f$ in  $\Lambda^{\alpha}_{\L}(\RR)$ is given by
 $$
 \|f\|_{\Lambda^{\alpha}_{\L}(\RR)}=\sup_{\substack{ x,y\in\RR\\
 x\not=y}}{|f(x)-f(y)|\over |x-y|^{\alpha}} +\sup_{x\in\RR} |\rho(x)^{-\alpha}f(x)|.
 $$
 Because of \eqref{e4.2}, this space $\Lambda^{\alpha}_{\L}(\RR)$  is in fact a proper subspace of  the classical Lipschitz space
 $\Lambda^{\alpha}(\RR)$ (see \cite{BHS2008, FJN, MSTZ, YSY}).
   Following \cite{BHS2008},  a locally integrable function $f$ in $\Real^n$ is in ${\rm BMO}_{\L}^{\alpha}(\RR)$, $\alpha> 0,$
  if there is a constant $C>0$ such that
\begin{eqnarray}\label{e4.3}
\int_{B}|f(y)-f_B|dy\leq Cr^{n+\alpha}
\end{eqnarray}
for every ball $B=B(x, r)$, and
\begin{eqnarray}\label{e4.4}
\int_{B} |f(y)|~dy\leq Cr^{n+\alpha}
	\end{eqnarray}
 for every  ball  $B=B(x, r)$  with $r\geq \rho(x)$. Define
 \begin{eqnarray*}
 \|f\|_{{\rm BMO}_{\L}^{\alpha}(\RR)} =\inf\Big\{ C: C \ {\rm satisfies }\  \eqref{e4.3} \ {\rm and}\  \eqref{e4.4}  \Big\}.
\end{eqnarray*}
It is proved in \cite[Proposition 4]{BHS2008}  that for $0<\alpha<1$,
$ {\rm BMO}_{\L}^{\alpha}(\RR)=\Lambda^{\alpha}_{\L}(\RR).$ And, we should note that, when $\alpha=0,$ ${\rm BMO}^\alpha_\L(\RR)={\rm BMO}_\L(\RR)=\Lambda^\alpha_\L(\RR).$

\begin{thm}\label{th4.1}
Suppose $V\in B_q$ for some $q\ge n,$  $\alpha\in (0, 1)$. We denote by ${\rm TMO^{\alpha}_\L}(\Real_+^{n+1})$  the class of all $C^1(\Real_+^{n+1})$-functions $u(x,t)$
of the solution of  ${\mathbb L}u=u_{t}+\L u=0$ such that
\begin{equation}\label{e4.5}
\|u\|^2_{{\rm TMO^{\alpha}_\L}(\Real_+^{n+1})}= \sup_{x_B, r_B}     r_B^{-(n+2\alpha)} \int_0^{r_B^2}\int_{B(x_B, r_B)}
\left\{t| \partial_t u(x,t) |^2+| \nabla_x u(x,t) |^2\right\}
{dx dt }  <\infty.
\end{equation}
Then we have
 \begin{itemize}
\item[(1)]  if $u\in {\rm TMO^{\alpha}_\L}(\Real_+^{n+1})$, then    there exist some $f\in \Lambda^{\alpha}_{\L}(\RR) $
and a constant $C>0$ such that $u(x,t)=e^{-t \L}f(x)$,
and
 $
\|f\|_{\Lambda^{\alpha}_{\L}(\RR) }\leq C\|u\|_{{\rm TMO^{\alpha}_\L}(\Real_+^{n+1})}.
 $

\item[(2)]  if $f\in \Lambda^{\alpha}_{\L}(\RR) $, then  $u(x,t)=e^{-t \L}f(x)\in {\rm TMO^{\alpha}_\L}(\Real_+^{n+1})$ with
 $
 \|u\|_{{\rm TMO^{\alpha}_\L}(\Real_+^{n+1})}\le C\|f\|_{\Lambda^{\alpha}_{\L}(\RR) }.
 $
 \end{itemize}
\end{thm}
Part (2) of Theorem~\ref{th4.1} is a straightforward result from  the following proposition.

\begin{prop}\label{prop4.2} Suppose $V\in B_q$ for some $q\ge n,$
$\alpha\in (0, 1)$ and  $f$ is a function such that
 $$\int_{\Real^n}\frac{\abs{f(x)}}{(1+\abs{x})^{n+\alpha+\varepsilon}}~dx<\infty$$
 for some $\varepsilon>0$. Then the following statements are equivalent:

 \begin{itemize}
\item[(1)]   $f
\in \Lambda^{\alpha}_{\L}(\RR) $;

 \item[(2)]  there exists a constant $C>0$ such that
 \begin{equation*}
\|t\partial_t e^{-t {\L}}f(x) \|_{L^\infty(\Real^n)}\le Ct^{\alpha\over 2};
 \end{equation*}

\item[(3)]   $u(x,t)=e^{-t {\L}}f(x)\in {\rm TMO^{\alpha}_\L}(\Real_+^{n+1})$
 with
 $
 \|u\|_{{\rm TMO^{\alpha}_\L}(\Real_+^{n+1})}\approx \|f\|_{\Lambda^{\alpha}_{\L}(\RR) }.
 $
 \end{itemize}
\end{prop}
\begin{proof}
{$\rm (1)\Rightarrow (2)$.}
Let $f\in \Lambda^{\alpha}_{\L}(\RR)$. One writes
\begin{align*}
   t\partial_t e^{-t {\L}}f(x)  &=  \int_{\Real^n}t\partial_t {\mathcal K}_t(x,z)\left(f(z)-f(x)\right)~dz
    +f(x)
	t\partial_t e^{-t {\L}}(1)(x)  \\
    & =  I(x)+II(x).
\end{align*}
From Lemma~\ref{le2.2}, we have
 \begin{eqnarray*}
 |I(x)|&\leq & C\|f\|_{\Lambda^{\alpha}_{\L}(\RR)} \int_{\Real^n}|t\partial_t {\mathcal K}_t(x,z)|\abs{x-z}^\alpha
     dz\\
     &\le&C\|f\|_{\Lambda^{\alpha}_{\L}(\RR)}
	 \int_{\Real^n}\frac{t}{t^{n+2\over 2}}e^{{-\frac{\abs{x-z}^2}{ct}}}\abs{x-z}^\alpha~dz\\
	 &\leq& C\|f\|_{\Lambda^{\alpha}_{\L}(\RR)}
	 \int_{\Real^n}\frac{t\abs{x-z}^\alpha}{\left(\sqrt t+\abs{x-z}\right)^{n+2}}~dz\\
 &\leq& Ct^{\alpha\over 2}\|f\|_{\Lambda^{\alpha}_{\L}(\RR)}.
\end{eqnarray*}
To estimate the term $II(x)$, we  consider two
cases.

\smallskip

\noindent
{\it Case 1:}   $\rho(x)\leq \sqrt t$. In this case we use Lemma~\ref{le2.2} (iii) to obtain
  \begin{eqnarray*}
 |II(x)|&\leq & \|f\|_{\Lambda^{\alpha}_{\L}(\RR)} \rho(x)^\alpha\big|t\partial_t e^{-t{\L}}(1)(x)\big|\\
 &\leq& Ct^{\alpha\over 2}\|f\|_{\Lambda^{\alpha}_{\L}(\RR)}.
\end{eqnarray*}

\smallskip

\noindent
{\it Case 2:}   $\rho(x)> \sqrt t$. Since  $\alpha<2-n/q,$ by  Lemma~\ref{le2.2} (iii),
 \begin{eqnarray*}
  |II(x)|&\leq&
C\|f\|_{\Lambda^{\alpha}_{\L}(\RR)}\rho(x)^\alpha \Big({\sqrt t\over \rho(x)}\Big)^{2-n/q}\\
&\leq&
C\|f\|_{\Lambda^{\alpha}_{\L}(\RR)}\rho(x)^\alpha\Big({\sqrt t\over \rho(x)}\Big)^{\alpha}\\
&=&
Ct^{\alpha\over 2}\|f\|_{\Lambda^{\alpha}_{\L}(\RR)},
\end{eqnarray*}
which, together with estimate of $I(x)$, yields
$ \|t\partial_t e^{-t {\L}}f(x) \|_{L^\infty(\Real^n)}\le Ct^{\alpha\over 2}\|f\|_{\Lambda^{\alpha}_{\L}(\RR)}.
$

\smallskip

{$\rm (2)\Rightarrow (3)$.}
For every ball $B=B(x_B,r_B)$,  we have
\begin{align*}
 \int_0^{r_B^2}\int_B  | t\partial_t e^{-t {\L}}f(x)|^2 {dx~dt\over t}
 &\leq C\|f\|^2_{\Lambda^{\alpha}_{\L}(\RR)}
\int_0^{r_B^2}\int_B t^{\alpha-1}~ {dt~dx}\\
&\leq Cr_B^{n+2\alpha}\|f\|^2_{\Lambda^{\alpha}_{\L}(\RR)},
\end{align*}
and hence   $u(x,t)=e^{-t {\L}}f(x)\in {\rm TMO_\L^\alpha}(\Real_+^{n+1}).$

\smallskip

 The proof of {$\rm (3)\Rightarrow (1)$} is a direct consequence of    \cite[Theorem 1.3]{MSTZ}.
 This completes the proof.
 \end{proof}

To prove part (1) of Theorem~\ref{th4.1}, we need some preliminary results.

\begin{lem}\label{le4.3}
Let $\alpha\in (0, 1).$
 For every  $u\in {\rm TMO}^{\alpha}_{\L}(\Real_+^{n+1})$ and
  every $k\in{\mathbb N}$, there exists a constant $C_{k, \alpha}>0$ such that
\begin{equation*}
\int_{\RR}{|u(x,{1/k})|^2\over 1+|x|^{2n+1}}  dx\leq C_{k, \alpha} <\infty,
\end{equation*}
and so $u(x, 1/k)\in L^2((1+|x|)^{-(2n+1)}dx)$. Hence for  all $k\in{\mathbb N}$, $e^{-t{\L}}(u(\cdot, {1/k}))(x)$ exists
everywhere in ${\mathbb R}^{n+1}_+$.
\end{lem}

\begin{lem} \label{le4.4} Let $\alpha\in (0, 1).$ For every  $u\in {\rm TMO}^{\alpha}_{\L}(\Real_+^{n+1})$ and for every $k\in{\mathbb N}$,
there exists a constant $C>0$ independent of $k$ such that
\begin{eqnarray*}
\|u(\cdot, 1/k)\|_{{\rm BMO}^{\alpha}_{\L}(\RR)}
&\leq& C\|u\|_{{\rm TMO}^{\alpha}_{\L}(\Real_+^{n+1})}.
\end{eqnarray*}
Hence for  all $k\in{\mathbb N}$, $u(x, 1/k)$ is uniformly bounded in  ${\rm BMO}^{\alpha}_{\L}(\RR)$.
\end{lem}

The proof of Lemma~\ref{le4.3} is   analogous to that  of Lemma~\ref{le3.1}. For Lemma~\ref{le4.4},
similar arguments as in Lemmas~\ref{le3.2} and \ref{le3.4} show  that
$u(x,t+1/k)=e^{-t{\L}}\big( u(\cdot, 1/k)\big)(x)$ satisfies
\begin{eqnarray}\label{e4.6} \hspace{1cm}
\sup_{x_B, r_B}     r_B^{-(n+2\alpha)} \int_0^{r_B^2}\int_{B(x_B, r_B)}  |t \partial_t e^{-t{\L}}\big( u(\cdot, 1/k)\big)(x) |^2
{dx dt\over t } \leq C\|u\|^2_{{\rm TMO}^{\alpha}_{\L}(\Real_+^{n+1})}
\end{eqnarray}
 for all $k\in{\mathbb N}.$
This,  together with  \cite[Theorem 1.3]{MSTZ}, shows that $u(x, 1/k)\in {\rm BMO}^{\alpha}_{\L}(\RR)$, and then
 the norm of $u(x, 1/k)$ in  the space ${\rm BMO}^{\alpha}_{\L}(\RR)$ is less than  the left hand side of \eqref{e4.6}.
 We omit the detail here.

 \smallskip

 \begin{proof}[Proof of part   (1)  Theorem~\ref{th4.1}] By  using Lemma~\ref{le2.6} for   $\partial_{t}u(x, t+{1/k})$ and
 a similar argument as in \eqref{e3.2}  we show that
  \begin{equation*}
 | \partial_{t}u(x, t) | \leq C t^{\alpha-1}  \|u\|_{{\rm TMO}^{\alpha}_{\L}(\Real_+^{n+1})},
  \end{equation*}
and hence if $t_1, t_2>0$,
 \begin{align*}
 |u(x, t_1)-u(x, t_2)|&=\abs{\int_{t_2}^{t_1}\partial_s u(x, s)ds}
 \leq C \|u\|_{{\rm TMO}^{\alpha}_{\L}(\Real_+^{n+1})}
 \abs{\int_{t_2}^{t_1}s^{\alpha-1}ds}\\
 &\leq C|t_1^{\alpha}-t_2^{\alpha}|
  \leq C|t_1-t_2|^{\alpha}
\end{align*}
since $0<\alpha<1.$ The family $\{u(x, t)\}$ is a Cauchy sequence as $t$ tends to zero and hence converges to some function $f(x)$
everywhere.

Now we apply Lemma~\ref{le4.4}, and note that for all $k\in{\mathbb N}$,
\begin{equation*}
\|u(\cdot, 1/k)\|_{{\rm BMO}^{\alpha}_{\L}(\RR)}\leq  C\|u\|_{{\rm TMO}^{\alpha}_{\L}(\Real_+^{n+1})}<\infty.
\end{equation*}
Letting $k$ tend to $\infty$,    we conclude
\begin{equation*}
\|f\|_{\Lambda^{\alpha}_{\L}(\RR)}\leq C
\|f\|_{{\rm BMO}^{\alpha}_{\L}(\RR)} \leq C\|u\|_{{\rm TMO}^{\alpha}_{\L}(\Real_+^{n+1})},
\end{equation*}
and hence
 $u(x,t)=e^{-t{\L}}f(x).$
This completes the proof of part (1) of Theorem~\ref{th4.1}.
 \end{proof}

\bigskip

{\bf Acknowledgments.}  C. Zhang was partly supported by the Natural Science Foundation of Zhejiang Province (Grant No. LY18A010006), the first Class Discipline of Zhejiang - A (Zhejiang Gongshang University- Statistics), and the National Natural Science Foundation of China (Grant No. 11401525). M. Yang was partly supported by the National Natural Science Foundation of China (Grant No. 11801236), Postdoctoral Science Foundation of China (Grant No. 2018M632593), Natural Science Foundation of Jiangxi Province for Young Scholars (Grant No. 20181BAB211001), the Postdoctoral Science Foundation of Jiangxi Province (Grant No. 2017KY23) and Educational Commission Science Programm of Jiangxi Province (Grant No. GJJ170345).

The authors would like to thank Dr. Huang Qiang for helpful discussions.


 \bigskip




\begin{thebibliography}{10}



 \bibitem{BHS2008} B. Bongioanni, E. Harboure and O. Salinas,
 {Weighted inequalities for negative powers of Schr\"odinger operators}.
 \textit{J. Math. Anal. Appl.}
 \textbf{348} (2008), 12--27.

  \bibitem{BHS2009} B. Bongioanni, E. Harboure and O. Salinas,
 {Riesz transforms related to  Schr\"odinger operators acting on BMO type spaces}.
 \textit{J. Math. Anal. Appl.} \textbf{357}(2009), 115--131.


 \bibitem{CMS} R. Coifman, Y. Meyer and E. Stein,
  Some new functions and their applications to harmonic analysis.
  {\it J. Funct. Analysis}, {\bf 62}(1985), 304--335.


 \bibitem {DKP} M. Dindos, C. Kenig and J. Pipher,
  BMO solvability and the $A_{\infty}$ condition for elliptic operators. \textit{J. Geom. Anal. } \textbf{21} (2011), 78--95.




\bibitem{DXY}X. Duong, J. Xiao and L. Yan, Old and new Morrey spaces with heat kernel bounds.
\textit{J. Fourier Anal. Appl.} \textbf{13} (2007), 87--111.

\bibitem   {DY1}  X. Duong and L. Yan, New function spaces of
{\rm BMO} type, the John-Nirenberg inequality, interpolation and applications,
{\it  Comm. Pure Appl. Math.} {\bf 58} (2005), 1375--1420.


\bibitem {DY2} X. Duong and L. Yan, Duality of Hardy and BMO spaces
associated with operators with heat kernel bounds. {\it J. Amer. Math. Soc.}
{\bf 18}(2005), 943--973.

\bibitem {DYZ}X. Duong, L. Yan and C. Zhang,
On characterization of Poisson integrals of Schr\"odinger operators with BMO traces. {\it J. Funct. Anal.} {\bf 266}(2014),  2053-2085.

\bibitem{DGMTZ} J. Dziuba\'nski, G. Garrig\'os, T. Mart\'inez, J.  Torrea and J. Zienkiewicz,
{$BMO$ spaces related to Schr\"odinger operators with potentials satisfying a reverse H\"older inequality}.
\textit{Math. Z.}
\textbf{249} (2005), 329--356.


\bibitem{DZ1999} J. Dziuba\'nski and J. Zienkiewicz, {Hardy space $H^1$ associated to Schr\"odinger operator with
 potential satisfying reverse H\"older inequality}.
\textit{Rev. Mat. Iberoamericana} \textbf{15} (1999), 279--296.

\bibitem{DZ2002} J. Dziuba\'nski and J. Zienkiewicz,
{$H^p$ spaces for Schr\"odinger operators}, in: Fourier Analysis and Related Topics \textbf{56}, Banach Center Publ., Inst. Math., Polish Acad. Sci.,
Warszawa, 2002, 45--53.


\bibitem{FJN}  E. Fabes, R. Johnson and U. Neri,
Spaces of harmonic functions representable by Poisson integrals of functions in BMO and $L_{p, \lambda}$.
\textit{ Indiana Univ. Math. J.}
\textbf{ 25} (1976), 159--170.


\bibitem{FN1} E. Fabes and U. Neri,
 Characterization of temperatures with initial data in BMO. \textit{ Duke Math. J. }
\textbf{42} (1975),  725-734.

\bibitem{FN} E. Fabes and U. Neri,   Dirichlet problem in Lipschitz domains with BMO data.
\textit{Proc. Amer. Math. Soc.}
\textbf{78} (1980), 33--39.



\bibitem{FS} C. Fefferman and E. Stein,   $H^p$ spaces of
 several variables. {\it Acta
Math.} {\bf 129} (1972), 137--195.



\bibitem{GJ} W. Gao and Y. Jiang, $L_p$   estimate for parabolic Schr\"odinger operator with certain potentials. \textit{J. Math. Anal. Appl.}  \textbf{310}  (2005), 128--143.

\bibitem{Ge} F. Gehring, The $L^p$-integrability of the partial derivatives of a quasiconformal mapping.
{\it Acta Math.},  {\bf 130} (1973), 265--277.





\bibitem  {HLMMY} S. Hofmann, G. Lu, D. Mitrea, M. Mitrea and L. Yan,
 Hardy spaces associated to non-negative
self-adjoint  operators satisfying Davies-Gaffney estimates.
{\it Memoirs of the Amer. Math. Soc.}, {\bf 214} (2011), no. 1007.


\bibitem{HMM} S. Hofmann, J. Martell and S. Mayboroda,  Layer potentials and boundary value problems for elliptic equations with complex $L^\infty$
   coefficients satisfying the small Carleson measure norm condition, \textit{Adv. Math.} \textbf{270} (2015), 480--564.



\bibitem{JX} R. Jiang, J. Xiao and D. Yang, Towards spaces of harmonic functions with traces in square Campanato spaces and their scaling invariants, \textit{Anal. Appl. (Singap.)} \textbf{14} (2016), 679--703.



\bibitem{Kurata} K. Kurata,
{An estimate on the heat kernel of magnetic Schr\"odinger operators and uniformly elliptic operators with non-negative potentials},
\textit{J. London Math. Soc. }
\textbf{62} (2000), 885--903.

\bibitem{MSTZ} T.  Ma, P.  Stinga, J.  Torrea and C. Zhang,
{Regularity properties of Schr\"odinger operators},
\textit{J. Math. Anal. Appl.}
\textbf{388} (2012), 817--837.


\bibitem{MSTZ2} T.  Ma, P.  Stinga, J.  Torrea and C. Zhang,
{Regularity estimates in H\"older spaces for Schr\"odinger operators via a T1  theorem},
\textit{Ann. Mat. Pura Appl.}
\textbf{193} (2014), 561--589.







\bibitem{Shen} Z. Shen,
{$L^p$ estimates for Schr\"odinger operators with certain potentials}.
\textit{Ann. Inst. Fourier (Grenoble)}
\textbf{45} (1995), 513--546.


\bibitem{Shen1999} Z. Shen,   On fundamental solution of generalized Schr\"odinger
operators.  {\it J. Funct. Anal.}   {\bf 167}
(1999),  521--564.

\bibitem{Song} L. Song, X. Tian and L. Yan, On characterization of Poisson integrals of Schr\"odinger operators with Morry traces, {\it Acta Math. Sin. (Engl. Ser.)}   {\bf 34} (2018),  787--800.


\bibitem{St1970} E. Stein,
\textit{Topics in Harmonic Analysis Related to the Littlewood-Paley Theory},
{Annals of Mathematics Studies} \textbf{63},
Princeton Univ. Press,
Princeton, NJ, 1970.




\bibitem{SW} E. Stein and G. Weiss,
\textit{Introduction to Fourier Analysis on Euclidean spaces},
Princeton Univ. Press,
Princeton, NJ, 1970.



\bibitem{TH} L. Tang and J. Han, $L_p$ boundedness for parabolic Schr\"odinger type operators with certain nonnegative potentials,
{\it Forum Math.} {\bf 23} (2011), 161--179.


\bibitem{WY} L. Wu and L. Yan, Heat kernels, upper bounds and Hardy spaces associated to the generalized Schr\"odinger operators, {\it J. Funct.
Anal.} {\bf 270} (2016), 3709-3749.

 \bibitem{YSY} W. Yuan, W. Sickel and D. Yang,
\textit{Morrey and Campanato meet Besov, Lizorkin and Triebel},
Lecture Notes in Mathematics, 2005. Springer-Verlag, Berlin, 2010.

  .



\end{thebibliography}
\end{document}